\documentclass[12pt,reqno]{article}
\usepackage{color,amsmath,amssymb,amsfonts,amsthm,latexsym,lscape,graphicx,ifthen,pdfpages,psfig,epsf}

\usepackage{fullpage}
\usepackage{float}

\newcommand{\seqnum}[1]{\href{http://oeis.org/#1}{\underline{#1}}}

\usepackage[colorlinks=true,
linkcolor=webgreen,
filecolor=webbrown,
citecolor=webgreen]{hyperref}
\definecolor{webgreen}{rgb}{0,.5,0}
\definecolor{webbrown}{rgb}{.6,0,0}

\begin{document}

\theoremstyle{plain}
\newtheorem{theorem}{Theorem}
\newtheorem{corollary}[theorem]{Corollary}
\newtheorem{lemma}[theorem]{Lemma}
\newtheorem{proposition}[theorem]{Proposition}

\newtheorem{definition}[theorem]{Definition}
\newtheorem{example}[theorem]{Example}
\newtheorem{conjecture}[theorem]{Conjecture}

\theoremstyle{remark}
\newtheorem{remark}[theorem]{Remark}

\newcommand{\Q}{\mathbb Q}
\newcommand{\Z}{\mathbb Z}
\newcommand{\C}{\mathbb C}
\newcommand{\N}{\mathbb N}

\def\glqq{,\,\!\!,}
\def\grqq{`\,\!`}
\def\eg{{\it e.g.},\,}
\def\Eg{{\it E.g.},\,}
\def\ie{{\it i.e.},\,}
\def\Ie{{\it I.e.},\,}
\def\viz{{\it viz}\ }
\def\etc{\,{\it etc.}\,}
\def\via{{\it via}\, }
\def\sspp{\,+\,}
\def\sspm{\,-\,}
\def\ssppm{\,\pm\,}
\def\sspeq{\,=\,}
\def\sspdef{\, :=\,}
\def\sspneq{\,\neq\,}
\def\sspkl{\,<\,}
\def\sspgr{\,>\,}
\def\sspgeq{\,\geq\, }
\def\sspleq{\,\leq\,} 
\def\spgeq{\ \geq\  }
\def\spleq{\ \leq\ } 
\def\sspmapsto{\,\mapsto\,}
\def\pn{\par\noindent}
\def\pb{\par\bigskip}
\def\ps{\par\smallskip}
\def\pbn{\par\bigskip\noindent}
\def\psn{\par\smallskip\noindent}
\def\sspentsp{\, \hat =\, }
\def\sspequiv{\,\equiv\,}
\def\sspnotequiv{\,\not\equiv  \,}
\def\sspmapsto{\,\mapsto\,}
\def\sspdiv{\,|\,}
\def\Beq{\begin{equation}}
\def\Eeq{\end{equation}}
\def\Eq#1{equation\, (#1)}
\def\Beqarray{\begin{eqnarray}}
\def\Eeqarray{\end{eqnarray}}
\def\sspin{\,\in\,}
\def\sspni{\,\ni\,}
\def\sspfed{\,=:\,}
\def\sspto{\,\to\,}
\def\notdiv{\,\nmid\,}
\def\rhs{right-hand side\,}
\def\lhs{left-hand side\,}
\def\dstyle#1{$\displaystyle #1 $}
\def\union{\cup}
\def\floor#1{\left\lfloor{#1}\right\rfloor}
\def\ceil#1{\left\lceil{#1}\right\rceil}
\def\abs#1{\vert {\,#1\,} \vert}
\def\Cases2#1#2#3#4{\left\{\begin{array}{ll}#1&\mbox{#2}\\ & \\#3&\mbox{#4}\end{array}\right.}
\def\binomial#1#2{{#1} \choose {#2}}
\def\sspdiv{\,|\,} 
\def\rprod{\prod\llap {\raise 8pt\hbox{$\rightarrow \thinspace$}}} 
\def\range#1#2{#1,\,\count15=#1 \advance\count15 by +1 \number\count15,\,...,\,#2}
\def\rangeinf#1{#1,\, \count16=#1 \advance\count16 by +1 \number\count16,\,...\,}
\font\smallcmr=cmr6
\def\cProd{\rlap{$\>\,\!${\raise2pt\hbox{\smallcmr c}}}\Pi}
\def\amodn#1#2{{#1}\,(\text{mod}\,{#2})}
\def\modstar#1{(\text{mod}^*\,{#1})}
\def\amodstarn#1#2{{#1}\,\modstar{#2}}
\def\modan#1#2{\text{mod}(#1,\,#2)}
\def\modstaran#1#2{\text{mod}^*(#1,\,#2)}

\def\repjn#1#2{_{#2}{\overline{#1}}}
\def\repstjn#1#2{_{#2}{\overline{#1}}^{\,*}}
\def\sspsimn{\,\buildrel n \over\sim\,}
\def\sspsimstarn{\,\buildrel {*n} \over\sim\,}
\def\sspsimstarn#1{\,\buildrel *#1 \over\sim\,}
\def\shouldbe{\,\buildrel ! \over \sspeq}

\begin{center}
\vskip 1cm{\LARGE\bf 
On the Equivalence of Three Complete Cyclic Systems of Integers  
}
\vskip 1cm
\large
Wolfdieter Lang \footnote{ \url{http://www.itp.kit.edu/~wl}} \\
Karlsruhe \\
Germany\\
\href{mailto:wolfdieter.lang@partner.kit.edu}{\tt wolfdieter.lang@partner.kit.edu}
\end{center}

\vskip .2 in

\begin{abstract}
The system of coaches by {\sl Hilton} and {\sl Pedersen}, the system of cyclic sequences of {\sl Schick}, and {\sl Br\"andli}-{\sl Beyne}, related to diagonals in regular $(2\,n)$-gons, and the system of modified modular doubling sequences elaborated in this paper are proved to be equivalent. The latter system employs the modified modular equivalence used by {\sl Br\"andli}-{\sl Beyne}. A sequence of {\sl Euler} tours related on {\sl Schick}'s cycles of diagonals is also presented.  
\end{abstract}
\section{Introduction}
{\bf A) Complete coach system $\bf \Sigma(b)$}
\psn
{\sl Hilton} and {\sl Pedersen} \cite{HP} found in the context of paper-folding the quasi-order theorem and introduced a complete system of coaches $\Sigma(b)$, for $b \sspeq 2\,n \sspp 1$, $n\sspin \N$. Each coach $\Sigma(b,\, i)$, for $i\sspin\{\range{1}{c(b)}\}$ consists of two rows of length $r(b,\, i)$. There is the upper row $A(b,\,i)$ with odd positive integers $a(b,\,i)_{j}$, for $j \sspeq \range{1}{r(b,\, i)}$, that are relatively prime to $b$, and the lower row $K(b,\,i)$ with positive integers $k(b,\,i)_{j}$, that are certain maximal exponents of $2$. All odd numbers from the upper rows of $\Sigma(b)$ constitute the smallest positive reduced residue system modulo $b$ for odd numbers (here called $RRSodd(b)$, given in \seqnum{A216319}). The recurrence for the $a_j$-numbers, defining also the $k_j$ numbers, are (the arguments $(b,\,i)$ are suppressed)  
\Beq\label{arec}
a_{j+1}\sspeq \frac{b\sspm a_j}{2^{k_j}}, \ \ \text{for}\ j\sspgeq 1,\ \text{and\ odd\ input}\ a_1 \sspkl \frac{b}{2},\ \text{with}\  \gcd(a_j,\,b)\sspeq 1,   
\Eeq
where $2^{k_j}$ is the maximal power of $2$ dividing $b\sspm a_j$, \ie $k_j$ is the $2$-adic valuation of $b \sspm a_j$. The length of a coach is determined by the cyclic condition $a(b,\,i)_{r(b,\,i)\sspp 1}\sspeq a(b,\,i)_1$ for the first time (the primitive period length is $r(b,\,i)$). The sum of the $k$ values of the lower row is identical for all coaches of $\Sigma(b)$:  \dstyle{k(b)\sspdef \sum_{j=1}^{r(b,\,i)}\,k(b,\,i)_j}, for each $i\sspin \{\range{1}{c(b)}\}$. The number of coaches $c(2\,n+1)$ is given in OEIS\cite{OEIS} \seqnum{A135303}$(n)$, for $n\sspgeq 1$, and  $k(2\,n+1)$, called the quasi-order of $\amodn{2}{(2\,n+1)}$, is given in \seqnum{A003558}$(n)$. 
\begin{example}(\cite{HP}, p.\,262)\label{Ex1}
\Beq
\Sigma(65)\sspeq \left| \begin{array}{c|ccc|ccc|ccccc}
 1 & 3 & 31 & 17 & 7 & 29 & 9 & 11 & 27 & 19 & 23 & 21 \\
 6 &  1 & 1 & 4 &1 & 2 & 3 & 1 & 1 & 1 & 1 & 2\end{array} \right|\, .
 \Eeq
\end{example} 
\pn
In the following a coach is written like $\Sigma(65,\,2)\sspeq [A(65,\,2),\,K(65,\,2)]\sspeq[[3,\,31,\,17],[1,\,1,\,4]]$.  
\pn
Here $c(65)\sspeq 4$, $k(65)\sspeq 6$, and the $r-$tuple is $\{r(65,1),...,\,r(65, 4)\}\sspeq \{1,\,3,\,3,\,5\}$.
\pn 
$k(b)$ satisfies the quasi-order theorem (\cite{HP}, p. 102 and p. 261)
\Beq \label{quoth}
2^{k(b)} \sspequiv \amodn{(-1)^{r(b,\,i)}}{b},\ \ \text{for\ each}\ \ i\sspin \{\range{1}{c(b)}\}\,. 
\Eeq
The $r-$tuples for $b\sspeq \{3,\,5,\,...,\,41\}$ are shown in \seqnum{A332434}, with $(-1)^{r(2\,n+1)}$ given in \seqnum{A332433}$(n)$, for $n \sspgeq 1$.
The coach theorem is (\cite{HP}, p. 262)
\Beq \label{coachth}
c(b)\,k(b)\sspeq \frac{\varphi(b)}{2},\ \ \text{for} \ \ b\sspeq 2\,n\sspp 1,\, n\sspin \N\,,
\Eeq
where {\sl Euler}'s totient is $\varphi\sspeq$\seqnum{A000010}. 
In the example \dstyle{4\cdot 6\sspeq 24\sspeq \frac{\varphi(65)}{2}}. 
\pbn
{\bf B) Complete cycle system $\bf SBB(b)$}
\psn
{\bf i)} In {\sl Schick}'s book \cite{Schick} a geometric algorithm in a unit circle is proposed which leads to periodic integer sequences with the recurrence relation for $\{q_j(b)\sspeq q_j(b;\,q_0)\}$ (in \cite{Schick} $p$ is used instead of the present odd $b$)
\Beq \label{SchickRec}
q_j(b) \sspeq b \sspm 2\,\abs{q_{j-1}(b)}, \ \ \text{for}\ \ j\sspin \N\,, 
\Eeq
with a certain odd input $q_0$, with $\gcd(q_0,\,b)\sspeq 1$, starting with $q_0(1)\sspeq (-1)^{\frac{b+1}{2}}$. This sequence of distinct odd numbers is periodic with length of the (principal) period $pes(b)$ ($pes$ stands in \cite{Schick} for {\it periode sp\'ectrale}, because the construction is called {\it Spektalalgorithmus}). The recurrence leads to identical period length for each input $q_0$. If not yet all odd members from $RRSodd(b)$, with $\#RRSodd(b)\sspeq\delta(b)\sspeq$\seqnum{A055034}$(b)\sspeq \varphi(b)/2$, are present then a new periodic sequence with the smallest missing member of $RRSodd(b)$ as input $q_0(2)$ is considered (where the sign is \dstyle{(-1)^{\frac{b+1}{2}}\,(-1)^{2\,q_0(2)-1}) }, \etc, until all elements of $RRSodd(b)$ appear. The number of such disjoint periodic sequences with odd members coprime to $b$ is in \cite{Schick}, Korollar, p. 14, called $B$. The recurrence shows that the primitive period length $pes(b)$ satisfies the same formula like $k(b)$ from the coach system, and it divides $\varphi(b)/2$. This leads to {\sl Schick}'s theorem, and $B(b)\sspeq c(b)$, the number of coaches, in accordance with the coach theorem \ref{coachth}. 
 \Beq \label{SchickId}
B(b)\,pes(b)\sspeq \frac{\varphi(b)}{2}\, .
\Eeq
See the tables 3.1 to 3.10 in \cite{Schick}, pp. 158 - 166 for these signed periodic sequences. These cycles will be called here $SBB(b,\,i)$  for $i\sspin \range{1}{B(b)}$ for the different inputs $q_0(b,\,i)$. The $BB$ in the $SBB$ notation comes from the following part {\bf ii)} on the {\sl Br\"andli}-{\sl Beyne} paper using unsigned sequences. The complete cycle system is then $SBB(b)$.
\begin{example}(\cite{Schick}, p. 161)\label{Ex2}
\Beqarray
SBB(65)&\sspeq& \{(-1,\, 63,\,-61,\,-57,\,-49,\,-33),\,(3,\,59\,-53, -41,\,-17,\,31),\,\nonumber  \\
&&(7,\,51,\,-37,\,-9,\,47,\,-29),\, (11,\,43\,\,-21,\,23,\,19,\,27)\}\, .
 \Eeqarray
\end{example} 
\pn
Here $B(65)\sspeq 4,\, pes(65)\sspeq 6$, and \dstyle{4\cdot 6 \sspeq \frac{\varphi(65)}{2}\sspeq 24}. Note that \dstyle{\delta(b)\sspeq \frac{\varphi(2\,b)}{2} \sspeq \frac{\varphi(b)}{2}}. 
\pbn
{\bf ii)} In the paper by {\sl Br\"andli}-{\sl Beyne} \cite{BB} unsigned {\sl Schick} sequences  are considered (we use the same symbols $q_j\sspeq q_j(b,\,i)$ but henceforth as positive integers)
\Beq \label{BBRec}
q_j(b,\,i) \sspeq \abs{b \sspm 2\,q_{j-1}(b,\,i)}, \ \ \text{for}\ \ j\sspin \N\,, 
\Eeq
with certain positive odd inputs $q_{0}(b,\, i)$, for $i\sspin \{\range{1}{B(b)}\}$, and $\gcd(q_{0}(b,\, i),\, b)\sspeq 1$, as explained above. The complete system of (sign-less) cycles is also named $SBB(b)$ (without risk of confusion).\psn
The numbers in these cycles correspond to certain length ratios diagonal/radius in a regular $(2\,b)$-gon, \viz \dstyle{d_j^{(2\,b)} \sspeq 2\,\sin\left(\frac{\pi}{2\,b}\,j\right)}, for $j\sspin \{\range{1}{2\,b\sspm 1}\}$. The diagonals connect the vertices $V_j^{(2\,b)}$, for $j\sspin\{\range{0}{2\,b}\}$, with $V_0^{(2\,b)}$ (Cartesian coordinates $(0,\, r)$, with radius $r$ of the circumscribing circle). The vertices are taken in the positive (counter-clockwise) sense, starting with the length ratio for the side $s(2\,b)/r \sspeq d_1^{(2\,n)}$.
\pn 
The diagonals in the upper half plane (including the real axis) are labeled with $j\sspin \{\range{1}{b}\}$. The (odd) numbers $k$ of each cycle stand for the labels of the diagonals $d_k^{(2\,b)}$.
\begin{example}\label{Ex3}
For $b\sspeq 17$ the two cycles are ($B(17)\sspeq 2$,\, $pes(17)\sspeq 4$)
\Beq
SBB(17,\,1) = (1,\,15,\,13,\,9), \ \ \text{and}\ \  SBB(17,\,2) = (3,\,11,\,5,\,7).
\Eeq
The elements of $RRSodd(17)$ are all odd numbers from 1 to 15.
\pn
The length ratios from the first cycle are (Maple \cite{Maple} 10 digits): $d_1^{(34)}\approx .1845367190 $, $d_{15}^{(34)}\approx 1.965946199$, $d_{13}^{(34)}\approx 1.864944459$, and $d_{9}^{(34)}\approx 1.478017835$. See {\sl Figure 1}.\pn
For $b\sspeq 11$, with $B(11) \sspeq 1$ and $pes(11)\sspeq 5$ see \cite{BB}, {\it Figure 1}, with the diagonals called $\sigma_{2\,n+1}$, for $n\sspin \{\range{1}{5}\}$.   
\end{example} 
\psn
Each of the $B(b)$ cycles, for each odd $b\sspgeq 3$, interpreted as length ratios for diagonals, suggests to consider a trail in the $(2\,b)$-gon by following the $pes(b)$ diagonals in the order of the cycle repeatedly, starting from node (vertex) $V^{(2\,b)}_0$ in the positive sense until a periodic sequence, an oriented {\sl Euler} tour, is completed. That this is always possible will be proved in the next section. Periodicity will be reached after $L(b,\,i)$ steps. For $b\sspeq 7$ see {\it Figure 2}. Because of periodicity one can start at any vertex which is reached in the tour starting with vertex $V^{(2\,b)}_0$. In this example all $14$ vertices are visited.
\psn
\parbox{16cm}{
{\includegraphics[height=8cm,width=.5\linewidth]{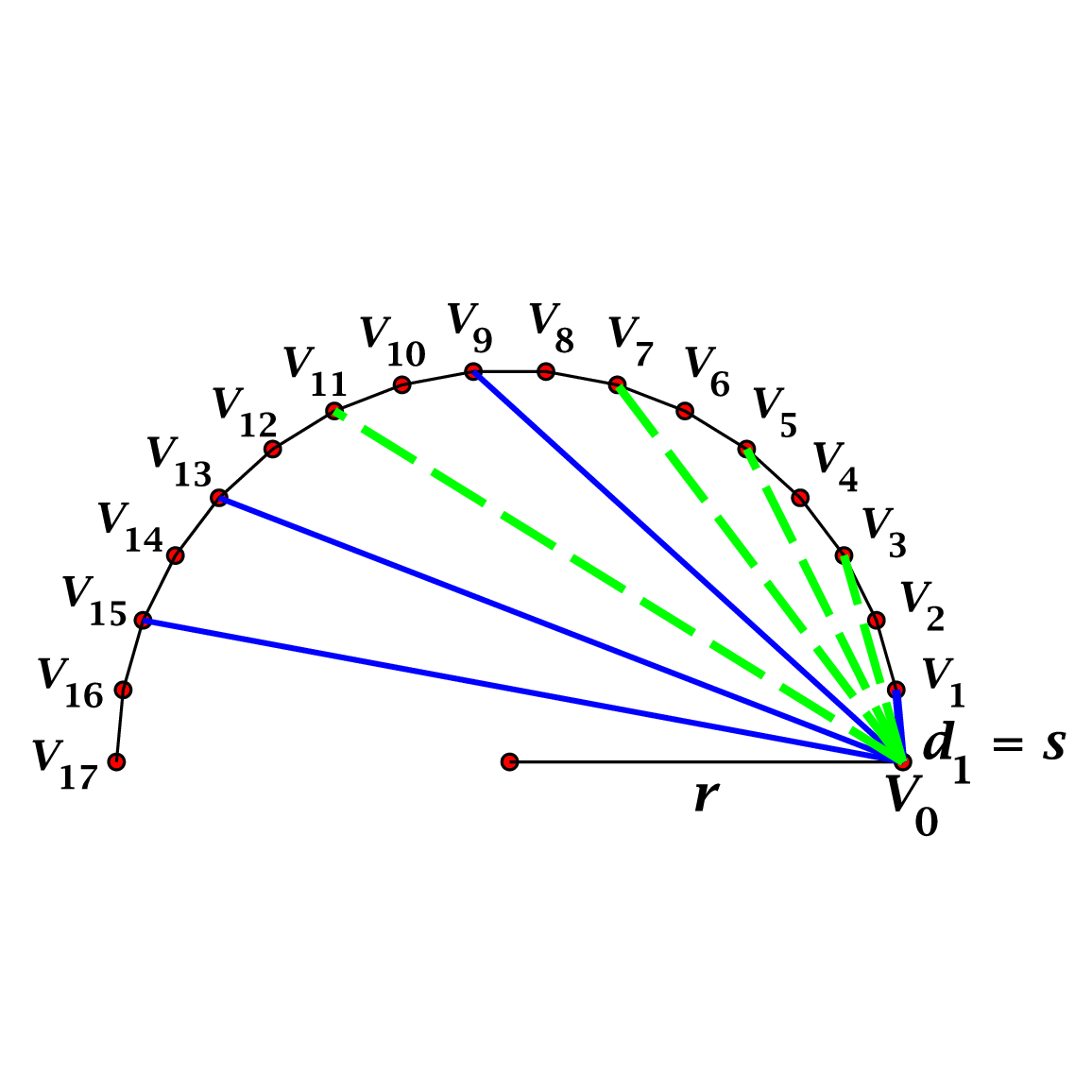}}
{\includegraphics[height=8cm,width=.5\linewidth]{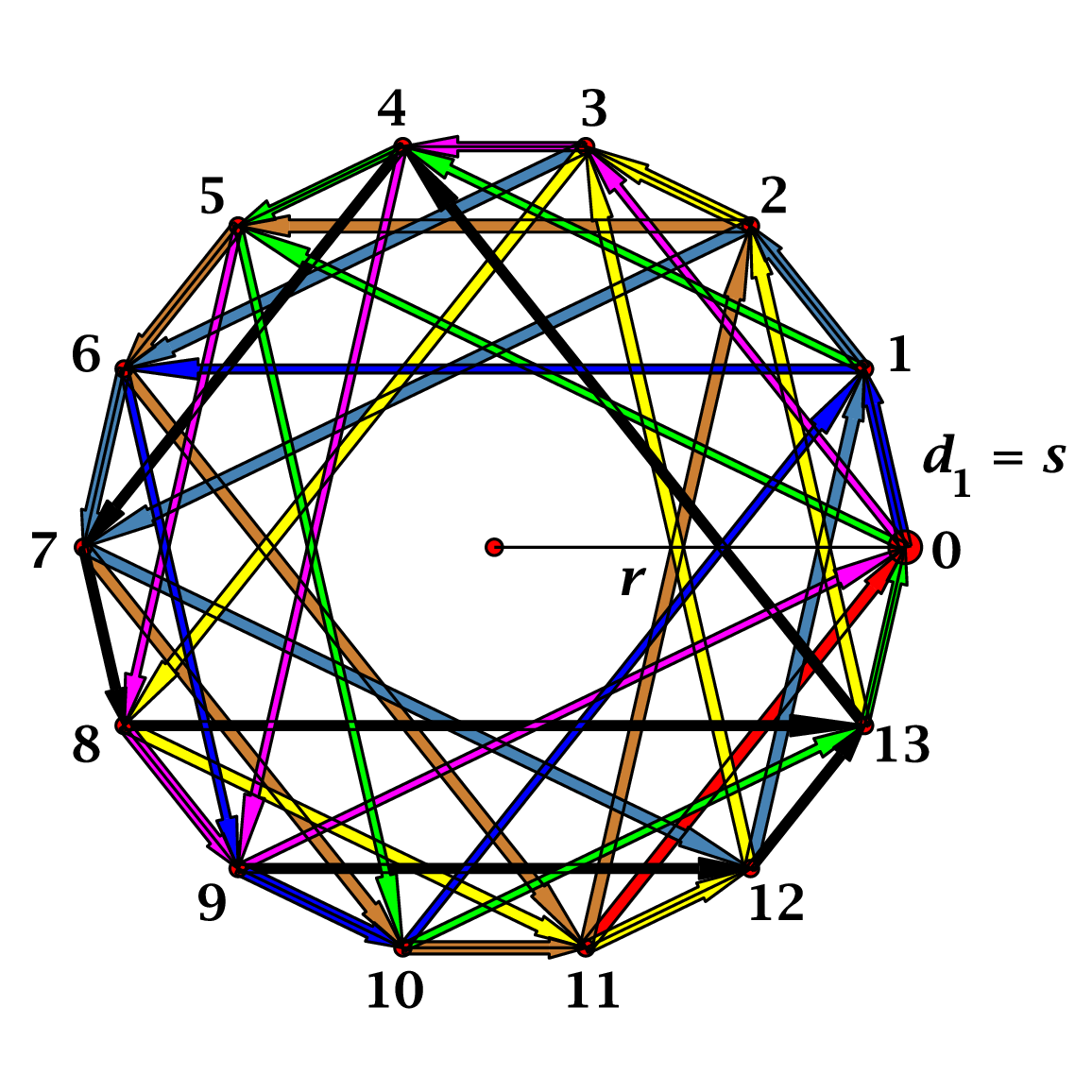}}
}
\psn 
\hskip 3.5cm  {\bf Figure 1}  \hskip 6cm  {\bf Figure 2}
\psn
{\bf Figure 1}: The diagonals in the $34$-gon for $b\sspeq 17$ for cycle $(1,\,15,\,13,\,9)$ (solid blue), and $(3,\,11,\,5,\,7)$ (dashed green).\psn
{\bf Figure 2}: The counter-clockwise {\sl Euler} tour for $b\sspeq 7$, with $14\cdot 3\sspeq 42$ arrows in the $14$-gon. Start at the vertex with label $0$. The color sequence for the arrows is blue, green, magenta, black, yellow, steel blue, and a final red arrow pointing back to label $0$. The colors are given in order to follow the trail without looking at the sequence of the $42$ vertex labels given later in section $2$. For an enlarged version see a link in \seqnum{A332439}.
\pbn 
Just for illustration, not for following the $68$ arrows in both cases in detail, see the {\it Figures 3} and {\it 4} for the two {\sl Euler} tours for $b\sspeq 17$. All $34$ vertices  are visited, in fact twice.
\psn
\parbox{16cm}{
{\includegraphics[height=8cm,width=.5\linewidth]{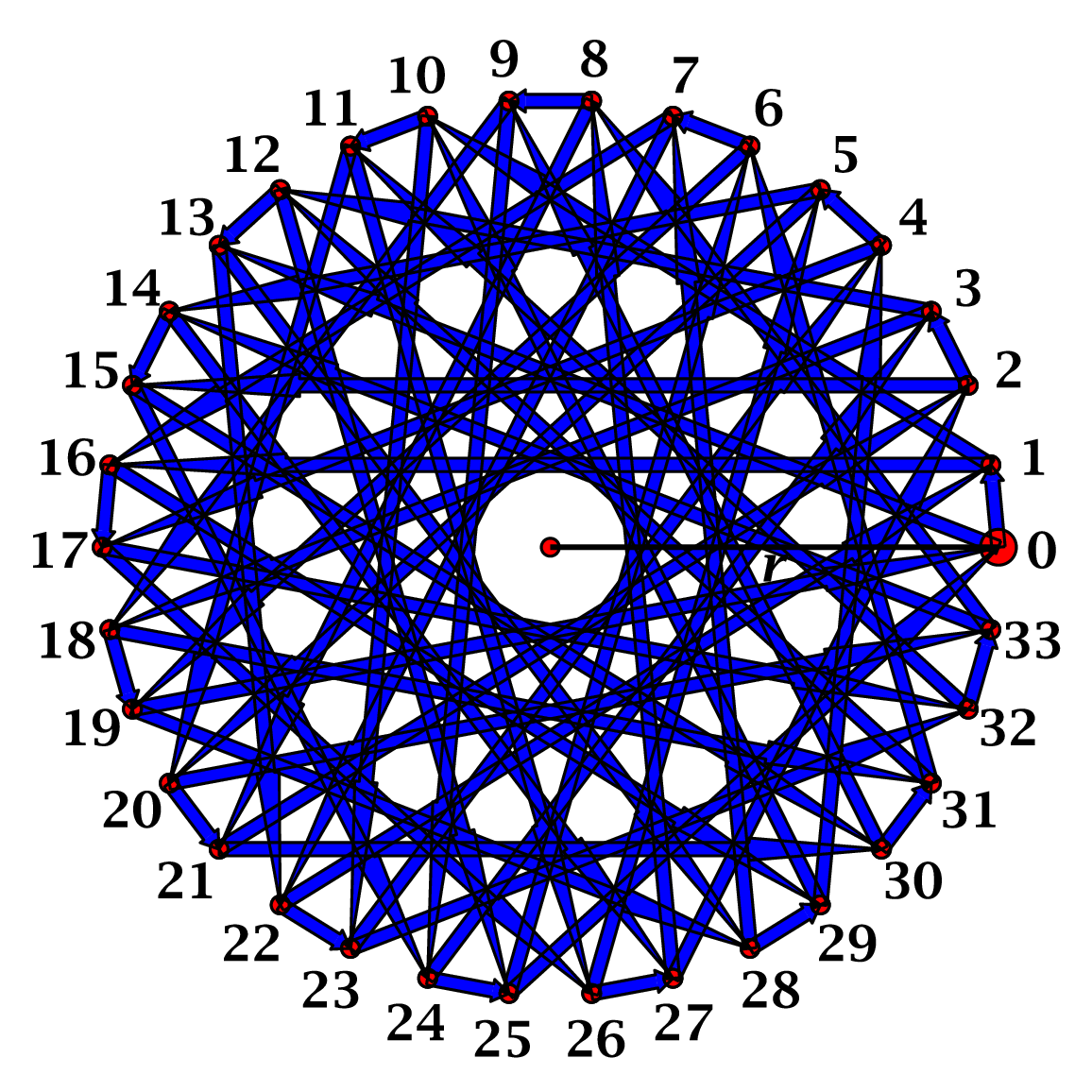}}
{\includegraphics[height=8cm,width=.5\linewidth]{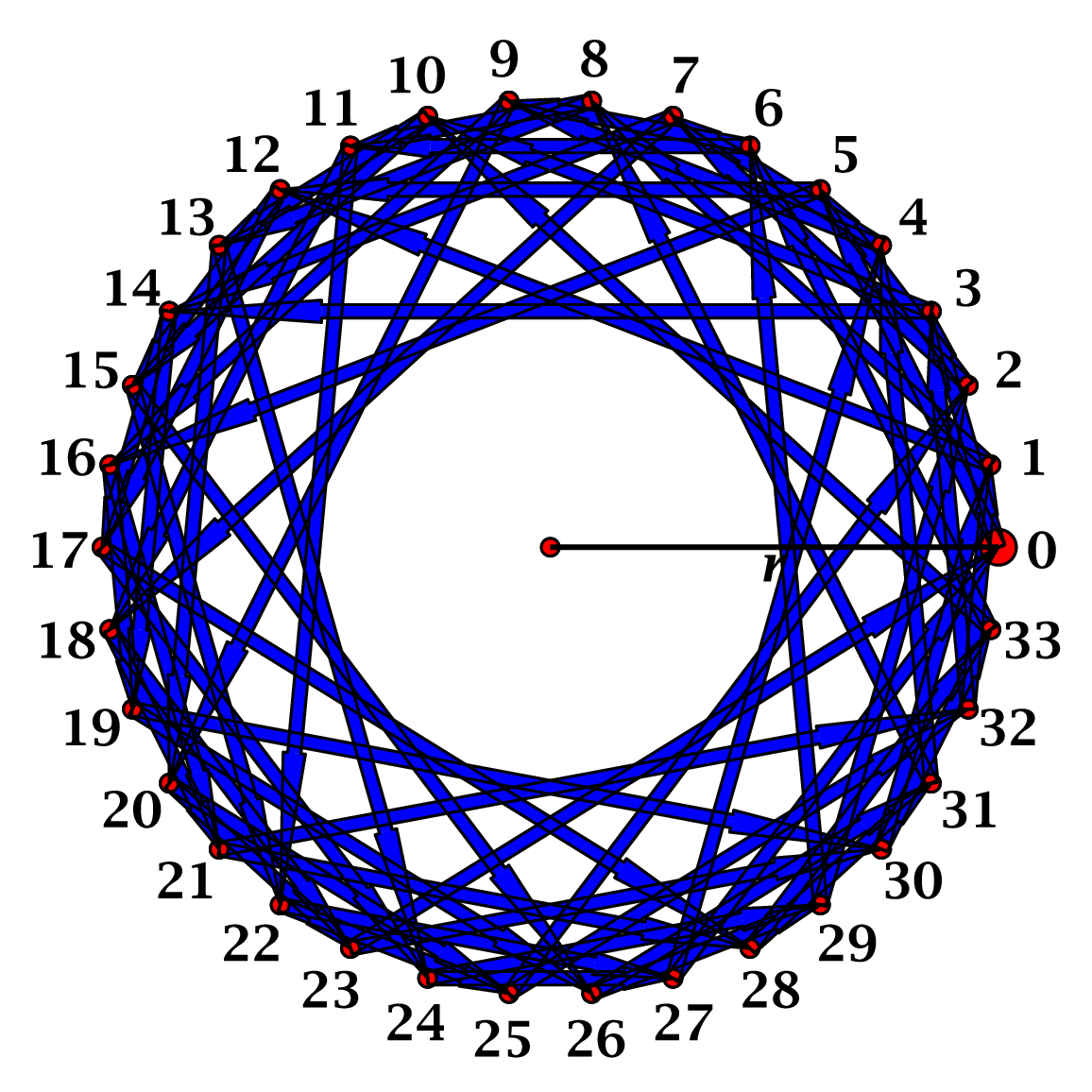}}
}
\psn 
\hskip 2.3cm  {\bf Figure 3}  \hskip 6.3cm  {\bf Figure 4}
\pbn
{\bf Figure 3}: $68$ arrows in the $34$-gon for the directed {\sl Euler} tour for $b\sspeq 17$ generated from the first cycle $(1,\,15,\,13,\,9)$ of diagonals.\psn
{\bf Figure 4}:  $68$ arrows in the $34$-gon for the {\sl Euler} tour for $b\sspeq 17$ generated from the second cycle  $(3,\,11,\,5,\,7)$ of diagonals.
\psn
\pbn 
The underlying digraph with $2\,b$ nodes (vertices) and $L(b,\, i)$ directed edges (arrows) which are trailed only once in the tour from, say, node $V^{2\,b}_0$, is simple (neither parallel arrows nor loops), but may not be regular (nodes may be of different order). Also, not all of the $2\,b$ nodes may be involved in the tour. If one keeps the remaining nodes one will have an unconnected digraph of connectivity number ${\cal N}(b)\sspm 1$. where the number of nodes of the {\sl Euler} tour is ${\cal N}(b)$, Alternatively, one can discard the remaining nodes in order to have a connected digraph. For example, for $b\sspeq 21$ with $B(21) = 1$ only $21$ nodes, not $42$, are involved in the tour, ${\cal N}(21)\sspeq 21$, and it is also an irregular graph. See {\it Figure 5}. 
\pbn
\hskip 4cm \parbox{16cm}{
{\includegraphics[height=8cm,width=.5\linewidth]{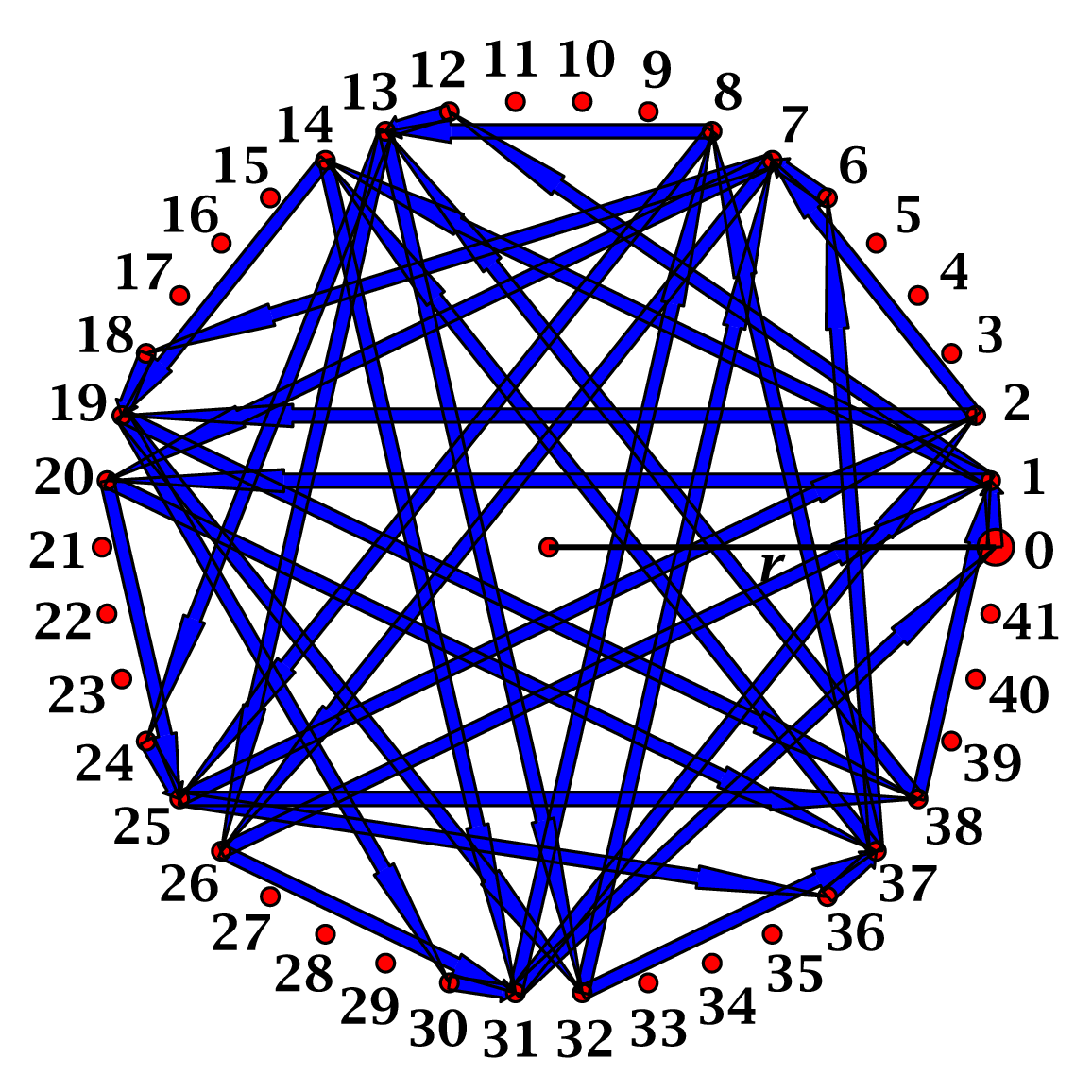}}
}
\psn
\hskip 7cm {\bf Figure 5}
\psn
{\bf Figure 5}: $42$ arrows in the $42$-gon for the directed {\sl Euler} tour for $b\sspeq 21$ generated from the cycle $(1,\,19,\,17,\,13,\,5,\,11)$. Only $21$ nodes are visited. The order of these nodes is $2,\,6,\,4$ for the labels $0,\,1,\, 2 \,(\text{mod}\,6)$, respectively.
\pbn
In the next section a theorem on the length of the primitive periods of these {\sl Euler} tours will be given. Then, in {\it Section 3}, the third complete system of cyclic sequences, the $MDS(b)$ system (for {\bf M}odified {\bf D}iagonal {\bf S}equences), for odd $b \sspgeq 3$, will be treated. This system is simpler than the mentioned ones, and will be proved to be equivalent to each of the other two systems in {\it Section 4}. A modified multiplicative modular relation will be employed, which has already been introduced in the {\sl Br\"andli} and {\sl Beyne} paper \cite{BB}. This modification of the doubling sequence with powers of $2$ has already been used in the proof of the quasi-order theorem in \cite{HP}, p. 103, \Eq{7.12}, and has also been proposed by {\sl Gary W. Adamson} in a comment, Aug 20 2019, to \seqnum{A003558}.
\pbn
\section{Schick sequences and Euler tours}
\psn
The cycles $SBB(b,\, i)$ for the $B(b) \sspeq$\seqnum{A135303}$((b-1)/2)$ different input values $q_0(i)$ ($B$ is identical with the coach number $c$), with period length $pes(b)\sspeq$\seqnum{A003558}$((b-1)/2)$ (identical with the quasi-order of $2$ modulo $b$ of the coach system, called $k(b)$), define by their odd elements length ratios diagonal/radius in a regular $(2\,b)$-gon, as explained above. Following a trail from, say, vertex $V^{(2\,b)}_0$ in the positive (counter-clockwise) sense, defined by a repetition of the diagonals according to the cycle numbers, leads, as will be proved in the theorem, to a positively oriented {\sl Euler} tour, named $ET(b, i)$, of length $L(b,\,i)$, with the given formula ~\ref{LengthEU}.
\psn
\begin{theorem}
\psn
\begin{enumerate}
\item The trail on a regular $(2\,b)$-gon, with odd $b\sspgeq 3$, starting at node (vertex) $V^{(2\,b)}_0$, with label $0$ and arrows (directed diagonals) following repeatedly the cycle $SBB(b,\, i)$ of length $pes(b)$, for any $i\sspin \{1,2,..., B(b)\}$, in the positive sense, will lead to a periodic sequence of node labels with  primitive length of the period $L(b,\,i)$. Therefore, this trail will be a counterclockwise oriented {\sl Euler} tour $ET(b, i)$.
\item The formula for the length of $ET(b, i)$ is
\Beq\label{LengthEU}
L(b,\,i) \sspeq \frac{2\,b\,pes(b)}{\gcd(SUM(SBB(b,\, i)),\, 2\,b)}\,\  ,
\Eeq
where $SUM(SBB(b,\, i))$ is the sum of the elements of the (unsigned) cycle $SBB(b,\, i)$.
\end{enumerate}
\end{theorem}
\psn
\begin{proof}\pn
\begin{enumerate}
\item
Starting from the node (vertex) with label $0$ the trail visits, in order, the nodes with labels given by the partial sums of the infinite periodic sequence $SBB(b, \,i)_{\infty} \sspeq {\text {repeat}}(SBB(b,\, i))$, evaluated modulo ${2\,b}$. After the first block of length $pes(b)$, starting with $0$, the second block of this length starts with $SUM(SBB(b,\, i))$. Each new block starts then with a multiple of  $SUM(SBB(b,\, i))$. Therefore the trail will become a periodic Euler tour if after some $m\sspeq m(b,\,i)$ blocks the next block would begin again with $0$ modulo $2\,b$. This means that $ SUM(SBB(b,\, i))\,m\sspequiv 0 \pmod{2\,b}$. Because this linear congruence has always at least one solution periodicity is guaranteed.
\item
The length $L(b,\,i)$ of the {\sl Euler} tour is then $m(b,\,i)\,pes(b)$. The solution of a linear congruence is known (see \eg \cite{Apostol}, Ch. 5.3, pp. 110 - 113), and in the present case it is the trivial solution $m\sspeq 0$ if $g\sspdef \gcd(SUM(SBB(b,\, i)),\,2\,b)\sspeq 1$, and if $g\sspgeq 2$ there is besides the trivial solution at least one other solution in $\{1, 2, ..., 2\,b\sspm 1\}$. If $g\sspeq 1$ the second appearance of the $0$ at the beginning of a block will be after $2\,b$ blocks, and $L(b,\,i)\sspeq 2\,b\,pes(b,\,i)$. If $g\sspgeq 2$ the smallest non-trivial solution is $(2\,b)/g$, and for this $m$ value periodicity appears for the first time.  
\end{enumerate}
\end{proof}
\psn
The sum of numbers of the unsigned {\sl Schick} cycles $ SBB(2\,n+1,\,q_0\sspeq 1)$ in the case of B(2\,n+1) = \seqnum{A135303}$(n) \sspeq 1$ are given in \seqnum{A333848}(n), for $n\sspgeq 1$. The corresponding numbers $\gcd(SUM(SBB(b,\, 1)),\, 2\,b)$ are given in \seqnum{A333849}. 
\pn
The $b$ numbers with $B(b)\sspgeq 2$ are listed increasingly in \seqnum{A333855}, and their $B$ numbers are in \seqnum{A333853}. For these numbers $b\sspeq 2\,n\sspp 1$ the sums of the cycles $ SBB(b,\, i)$, for $i\sspin \{1, 2, ..., B(b)\}$, are given in the table \seqnum{A333850} (where $k$ is used instead of $i$). The corresponding $\gcd(  SUM(SBB(b,\, i)),\, 2\,(2\,n+1))$ values are given in table \seqnum{A333851}.
\psn
Note that the $\gcd(SUM(SBB(b,\, i)),\,2\,b)$ values are not independent of $i$. See \seqnum{A333851} for the first cases with $b \sspeq 65,\,133,\,...\, $.   
\psn
\begin{example}\label{Ex4} {\bf Euler tour $\bf{ET(7,\,1)}$} 
\psn
$b\sspeq 7,\, B\sspeq 1,\,pes(7)\sspeq 3$, $ SBB(7,\,1)\sspeq [1,\,5,\,3]$. $SUM(SBB(7,\,1))\sspeq 9$, $\gcd(9,\,2\cdot 7)\sspeq 1$, length \dstyle{L(7,\,1) \sspeq \frac{2\cdot7\cdot 3}{1}\sspeq 42}
\psn
$ ET(7,\,1)\sspeq [0,\,1,\,6,\,9,\,10,\,1,\,4,\,5,\,10,\,13,\,0,\,5,\,8,\,9,\,0,\,3,\,4,\,9,\,12,\,13,\,4,$ 
\pn
\hskip 2.45cm $7,\, 8,\,13,\,2,\,3,\,8,\,11,\,12,\,3,\,6,\,7,\,12,\,1,\,2,\,7,\,10,\,11,\,2,\,5,\,6,\,11]$\,. 
\psn
The corresponding digraph is shown in {\sl Figure 2}, and enlarged as a link in \seqnum{A332439}. There are $42$ arrows, and all $14$ nodes are visited thrice (a regular graph of order $6$). 
\end{example}
\section {\bf Interludium on Schick's cycles for primes}
\psn
Special attention is paid in {\sl Schick}'s book to prime $b$ (remember that he uses $p$ for odd numbers $\sspgeq 3$). The following theorem will show that \seqnum{A268923} gives in fact the odd prime numbers with $B\sspgeq 2$. This will be done by proving that the complement relative to the set of odd primes is \seqnum{A216371} (the odd primes with $B\sspeq 1$). In this proof the identity,  equation ~\ref{SchickId} is used, and $2\,pes(prime(n)) \sspeq order(2,\,3\,prime(n))$, for $n\sspgeq 2$, is derived.  
\psn
\begin{theorem}
\psn
\begin{enumerate}
\item \label{order} $2\cdot order(4, prime(n)) \sspeq order(2,\,3\,prime(n))$, \ for \  $n \sspgeq 2$.
\item \label{pesprime} $pes(prime(n)) \sspeq order(4, prime(n))$,\ for \  $n \sspgeq 2$.
\item \label{Th2.3} The set of all odd primes $p$ such that all odd primes $q$, with $q\sspkl p$, satisfy \pn
\dstyle{\frac{\varphi(p\,q)}{2}\, {\Big /} \,order(2,\,p\,q) \sspgr 1} is equal to the set of odd primes satisfying $B(prime) \sspgeq 2$.
\end{enumerate}
\end{theorem}
\psn
\begin{proof}
\pn
\begin{enumerate}
\item 
It will be proved for odd primes $p$ that $2^{2 \cdot order(4,\, p)} \sspm 1\sspeq 3\,k\,p$, with integer $k$, and that this exponent is the least one $\sspgeq 1$. Clearly $4^{order(4,\, p)} - 1 \sspeq k'\,p$ with some integer $k'$ by definition of $order(4, p)$, and this exponent is the least one $\sspgeq 1$. Also,  $2^{2\,n} - 1 \sspeq (1\sspp 3)^n - 1 \sspequiv 0 \pmod{3}$, certainly for any positive $n$, by the binomial theorem. Hence $k'\sspeq 3\,k$. See \seqnum{A082654}$(n) \sspeq order(4, prime(n))$, for $n\sspgeq 2$.     
\item 
The {\sl Jonathan Skowera} formula, Jun 29 2013, in \seqnum{A003558} shows that \seqnum{A003558}$(\hat n)\sspeq pes(2\,\hat n +1)$, for $\hat n\sspgeq 1$. Hence for $prime(n)\sspeq 2\,\hat n \sspp 1$, with $\hat n\sspeq \hat n(n)$, one has $pes(prime(n)) \sspeq$\seqnum{A003558}$(\hat n(n))$, for $ n \sspgeq 2$.\pn 
The direction $\Rightarrow$ is then trivial from the definition of \seqnum{A003558}$(\hat n(n))$ by just squaring.\pn
The direction $\Leftarrow$ with \seqnum{A082654}$(n)$, for $n\sspgeq 2$, 
means that $(2^{k(n)}\sspp 1)\,(2^{k(n)}\sspm 1))\sspequiv 0 \pmod{prime(n)}$. But this implies $prime(n)$ divides one of the factors, hence $2^{k(n)}\sspequiv\pm 1 \pmod{prime(n)}$, for $n\sspgeq 2$, with minimal $k(n) \sspgr 1$; this is the definition from \seqnum{A003558}.
\item
It is proved that the complement of \seqnum{A268923} relative to the set of odd primes is \seqnum{A216371}. For the complement one tries to find all odd primes $p$ such that an odd prime $q\sspkl p$ exists with \dstyle{\frac{\varphi(p\,q)/2}{order(2,\,p\,q)} \sspeq 1}.
\pn
As such a prime $q\sspeq 3$ is chosen. From the {\sl Schick}'s theorem eq.~\ref{SchickId} one knows that \dstyle{B(p) \sspeq \frac{p\sspm 1}{2\,pes(p)}}, for $p\sspgeq 3$. With parts {\it 1} and {\it 2}, just proved, this becomes \dstyle{B(p) \sspeq \frac{p\sspm 1}{order(2,\,3\,p)}}.\pn   
In addition a corollary of {\it Theorem} 64,\, p. 106, by {\sl Nagell} \cite{Nagell} (or the Theorem 10.8, p. 211,  of {\sl Apostol} \cite{Apostol}) is needed, namely the improvement of the {\sl Euler}-{\sl Fermat} theorem for certain non-exceptional numbers (here $p\,q$, with $q\sspkl p$): $2^{\varphi(p\,q)/2}\sspequiv 1 \pmod{p\,q}$, where $\varphi\sspeq$\seqnum{A000010}, {\sl Euler}'s totient function. It is then clear that $order(2,\, p\,q)$ has to divide $\varphi(p\,q)/2$. For $q\sspeq 3$ this becomes \dstyle{\frac{p\sspm 1}{order(2,\,3\, p)} \sspgeq 1} for all $p\sspgeq 5$.
\pn 
From application of {\sl Schick}'s theorem \Eq{6} it is clear for which odd primes $\sspgeq 5$ equality holds, namely for those with $B(p)\sspeq 1$. Hence the complement of \seqnum{A216371} is $3$, and the primes $p$ with $B(p)\sspeq 1$, which is \seqnum{A216371}. Therefore \seqnum{A268923} gives the odd primes $p$ with $B(p)\sspgeq 2$. 
\end{enumerate}
\end{proof}
\section{Complete cycle system of modified modular doubling sequences $\bf MDS(b)$}
\psn
{\bf A) Modified modular congruence $\bf \modstar{n}$ }
\psn
This modified modular congruence considered by {\sl Br\"andli} and {\sl Beyne} can be  formulated based on ordinary reduced residue systems modulo $n$, whose smallest non-negative version is given by the set named RRS(n) shown in \seqnum{A038566} (but with  ${0}$ for $n\sspeq 1$). $RRS(n)\sspeq \{r_1,\, r_2,\,...\,r_{\varphi(n)}\}$, with $r_j$ representing the residue class $\repjn{r_j}{\,n}$, and $ \varphi(n)\sspeq \#RRS(n)\sspeq$\seqnum{A000010}$(n)$, for $n\sspin \N $. \Eg $RRS(4)\sspeq \{1,\,3\}$, with representative $\repjn{1}{4}$ consisting of all integers $a\sspeq 1\sspp k\,4$, with $k\sspin \Z$, \ie \, $a\sspequiv \amodn{1}{4}$, and similarly for $\repjn{3}{4}$. Thus $RRS(4)$ represents all odd integers (the union of the sets congruent to $1$ and $3$ modulo $4$).
\begin{definition} \label{modstar}
The smallest non-negative reduced residue system $\text{mod}^*\,n$ is denoted by $RRS^*(n)$, and this is $RRS^*(1)\sspeq \{0\}$, $RRS^*(2)\sspeq \{1\}$, and $RRS^*(n)$ consists of the elements of the first half of $RRS(n)$, for $n\sspgeq 3$.  \pn
Each element $r^*_j\sspin RRS^*(n)$ represents the reduced residue class denoted by $\repstjn{r_j}{\,n}$, and these classes are given in terms of ordinary ones by 
\Beqarray
 &&\repstjn{0}{1}\sspeq \repjn{0}{1}\,\ \ \  \repstjn{1}{2}\sspeq \repjn{1}{2},\, \nonumber \\
 && \repstjn{r_j}{n}\sspdef   \{\repjn{r_j}{n}\}\, \union\, \{\repjn{(-r_j)}{n}\},\ \ {\text {for}} \ \ j \sspeq \range{1}{\frac{\varphi(n)}{2}}, \ \  \text{and}\ \ n\sspgeq 3\,.
\Eeqarray
\end{definition}
\pn
The number of elements of $RRS^*(n)$ is $\#RRS^*(n)\sspfed \sigma(n)\sspeq $\seqnum{A023022}$(n)$, but with extra \seqnum{A023022}$(1) \sspeq 1$. Note that the $\text{mod}^*$ residue classes represent together the same set of integers as the ones of $RRS(n)$, but $RRS^*(n)\sspeq RRS(n)$ only for $n\sspeq 1$ and $n\sspeq 2$. \psn
\Eg, $n\sspeq 4$: $RRS^*(4)\sspeq \{1\}$, and the reduced residue class $\repstjn{1}{4}$ consists of the union of the ordinary reduced residue classes $\repjn{1}{4}$  and $\repjn{3}{4}$, hence $RRS^*(4)\sspneq RRS(4)$, but represents the same integers as $RRS(4)$, namely the odd ones (see above). \psn
It is elementary that the elements of $RRS^*(2\,m)$ are only odd (and reduced) integers, and $RRS^*(2^m)$ represents all odd numbers, for each $m\sspin \N$.\psn
If in $RRS^*(2\,m\sspm 1)$, with  $m\sspin \N$, all even integers are discarded, the remaining subsets of the residue classes could be called $RRS^*odd(2\,m\sspm 1)$. But it is clear that $RRS^*odd(2\,m\sspm 1)$ represents the same odd integers as $RRS^*(2\,(2\,m\sspm 1))$. 
\psn
The equivalence relation $a \sspsimstarn{n}\, b$ for integer numbers $a$ and $b$, coprime to $n\sspin \N$, is defined by $a\sspin \repstjn{r_j}{\,n}$ and $b\sspin \repstjn{r_j}{\,n}$, with $r^*_j\sspin RRS^*(n)$, for some $j\sspin \{\range{1}{\sigma(n)}\}$. This is based on the fact that the restricted residue classes ${\text mod}^*n$ are pairwise  disjoint and cover the integers coprime to $n$. 
\psn 
The notation $\modstaran{a}{n}$ gives $r^*_j\sspin RRS^*(n)$, if $a\sspin \repstjn{r_j}{\,n}$.\pn
It is defined by
\Beq
\modstaran{a}{n}\sspdef \Cases2{\modan{a}{n}}{${\text if}\ \modan{a}{n}\sspleq \frac{n}{2}$\,,}{\modan{-a}{n}}{${\text if}\ \modan{a}{n} \sspgr \frac{n}{2}$\,,}
\Eeq
where $\modan{a}{n}\sspin \{\range{0}{n-1}\}$ ($0$ can appear only for $n\sspeq 1$ because $a$ is reduced).
\psn 
\begin{example}\label{Ex5}
\pn
\begin{enumerate} 
\item\  $\modstaran{17}{9}\sspeq 1\ \text{because}\ \ \modan{17}{9}\sspeq 8\ \ \text {and} \ \  \modan{-8}{9}\sspeq 1.$ \nonumber 
\item \ \ $\modstaran{4}{9}\sspeq 4\ \text{because}\ \ \modan{4}{9}\sspeq 4 \sspkl 4.5\,.$\nonumber 
\item \ \ $\modstaran{5}{9}\sspeq 4\ \text{because}\ \ \modan{5}{9}\sspeq 5 \sspgr 4.5\,.$\nonumber  
\end{enumerate}
\end{example}
\pn
Because $\text{mod}^*\, n$ uses numbers coprime to $n$ it is not additive for $n\sspgeq 2$, contrary to $\text{mod}\,n $. If $\gcd(a,\,n)\sspeq 1\sspeq \gcd(b,\,n)$ then $\gcd(a\sspp b,\,n)\sspneq 1$ in general if $n \sspgeq 2$. Because $\modstaran{a}{n}\sspeq \modstaran{-a\sspp n}{n}$, even if $\gcd(a\sspp b,\,n)\sspeq 1$ (\eg $n\sspeq 9$, $a\sspeq 1,\, b \sspeq 4$) $\text{mod}^*$ cannot be well defined (\eg $4\sspsimstarn{9}\, 5$, and $\gcd(1\sspp 5,\,9)\sspeq 3$). 
\psn
However, $\text{mod}^*$, like $\text{mod}$, is multiplicative. This is based on the lemma that if  $\gcd(a,\, n)\sspeq 1$ and $\gcd(b,\, n)\sspeq 1$ then $\gcd(a\, b,\, n)\sspeq 1$ (see, \eg the simple proofs in \cite{Burton}, p. 30, Exercise 16 (a),\,or \cite{NZ}, p. 8, Satz 1.8). Also, $\gcd(a\,(-a\sspp n),\,n)\sspeq 1$ provided $\gcd(a,\,n)\sspeq 1$.
\pn
That each reduced residue system  $\text{mod}^*n$, with $\sigma(n) \sspgeq 2$, forms an {\sl Abel}ian multiplicative group of degree $\sigma(n)$ follows from $\text{mod}\,n$ properties. The identity element is the class congruent to $\repstjn{1}{\,n}$. If the system $RRS^*(n)$ is used this is the element $1$. That a unique inverse element exists is based on a theorem on the unique solution of the linear congruence $a\,x\sspequiv \amodn{b}{n}$ if $\gcd(a,\,n)\sspeq 1$ (see, \eg \cite{NZ}, p. 41, Satz 2.13, \cite{Apostol}, p. 114, Theorem 5.20), which generalizes to the $\text{mod}^*$ case. The cases with  $\sigma(n)\sspeq 1$  are trivial one element groups.      
\psn
In \cite{BB} this multiplicative group of representatives of $RRS^*(n)$ has been considered and named $G^*_n$. For $\sigma(n)\sspgeq 2$ this is a direct product of cyclic groups $C_k$ of order k, because it is a finite Abelian group (see \eg \cite{Speiser}, Satz 43 and Satz 42 p. 49 and 47). 
\psn
If $n$ is a product of two odd primes the smallest noncyclic group appears for $n\sspeq 65\sspeq 5\cdot 13$ with the cycle structure $12_2\,6_2$ (meaning $2$ cycles of length $12$ and also of length $6$, not regarding subcycles), namely $(3,\, 9,\, 27,\, 16,\, 17,\, 14,\, 23,\, 4,\,12,\, 29,\, 22,\, 1)$, $(11,\, 9,\, 31,\, 16,\, 19,\, 14,\, 24,\, 4,\,$\pn
$ 21,\, 29,\, 6,\,1)$, $(2,\, 4,\, 8,\, 16,\, 32,\, 1)$ and $(7,\, 16,\, 18,\, 4,\, 28,\, 1)$. See \cite{BB}, Theorem 3, p. 10, for the criterion on cyclic $\text{mod}^*$ groups for odd $n$. $G^*_{65}$ is the group $C_4\times C_3\times C_2$. For the cycle graph see \cite{WL}, Figure 4, the 7th plot, where cycle groups are called $Z_k$ (as a side remark note that the modified modular congruence $\text{Modd}$ used there in the context of regular $n$-gons and diagonal/side ratios is related to a set of representatives, called there ${\cal M}(n)$ (\Eq{65}) that is different from $RSS^*(n)$ considered here).
\psn   
In a further paper {\sl Br\"andli} and {\sl Beyne}, with two collaborators \cite{BBLL}, studied polynomials which have as {\sl Galois} group $G^*_n$, now called $\Z^{*/2}_n \sspeq \Z^*_n/\langle n-1\rangle$. They call these polynomials $\Psi^{re}_n(x)$ ({\it re} stands for reduced), and give the zeros in terms of \dstyle{s_k(n)\sspdef 2\,\cos\left(\frac{2\,\pi\,k}{n}\right)}, for $k$ values from $RRS^*(n)$, and  $n\sspgeq 1$. This can be rewritten in terms of the algebraic number $\rho(n)\sspdef 2\,\cos(\pi/n)$ of degree $\delta(n) \sspeq$\seqnum{A055034}$(n)$, which is fundamental for regular $n-$gons (see \cite{WL}), as $s_k(n)\sspeq R(2\,k,\,\rho(n))$, with the monic {\sl Chebyshev} polynomials $R$ shown in \seqnum{A127672}. The minimal polynomials of $\rho(n)$, called $C(n,x)$ in \cite{WL} (see also Table 2 there) are used in order calculate the {\sl Galois} polynomials $\Psi^{re}_n(x)$ modulo $C(n)$ (in order to reduce powers of $\rho(n)$ with exponents larger than $\delta(n)\sspm 1$).\pn
These polynomials are then (the use of $C(n,\,x\sspeq\rho(n))$ is indicated by the vertical bar)  
\Beq
\Psi^{re}_n(x)\sspeq \prod_{j=1}^{\sigma(n)} \left(x \sspm  R(2\,RRS^*(n)_j,\,\rho(n))\right) \Big |_{C(n,\,\rho(n))\sspeq 0}\,,\ \ \text{for}\ \ n\sspin \N\,.
\Eeq
It is not surprising that these are the minimal polynomials of the algebraic number \pn 
\dstyle{2\,\cos\left(\frac{2\,\pi}{n}\right)\sspeq R(2,\,\rho(n))\sspeq \rho(n)^2\sspm 2} of degree $\sigma(n)$, and they are given in \seqnum{A232624}, where they are named $MPR2(n,\,x)$.
\psn
\begin{example}\label{Ex6}
\Beqarray
&&\Psi^{re}_{65}(x)\sspeq MPR2(65,\,x)\sspeq 1\,-\,12\,x\,-\,180\,x^2\,-\,101\,x^3\,+\,2085\,x^4\,+\,1802\,x^5\,-\,9126\,x^6\, \nonumber \\
&&-\,7168\,x^7\,+\,20886\,x^8+\,13653\,x^9\,-\,28667\,x^{10}\,-1\,5001\,x^{11}\,+\,25284\,x^{12}  +\,10282\,x^{13}\, \nonumber \\
&&-\,14822\,x^{14}\,-\,4540\,x^{15}\,+\,5832\,x^{16}+\,1292\,x^{17}\,-\,1521\,x^{18}\,-\,229\,x^{19}\,+\,252\,x^{20}\,+\,23\,x^{21}\, \nonumber \\
&&-\,24\,x^{22}\,-\,x^{23}\,+\,x^{24}\,.
\Eeqarray
\end{example}
\pbn
{\bf B) Complete $\bf MDS(b)$ system}
\psn
Instead of the two lines  $A(b,i)$  and $K(b,i)$, with $b\sspeq 2\,n\sspp 1$, $n\sspin \N$, $i\sspin \{\range{1}{c(b)}\}$ and length $r(b.\,i)$ in the complete coach system $\Sigma(b)$ of {\sl Hilton} and {\sl Pedersen} \cite{HP} from section {\bf A)} of the {\it Introduction}, and the periodic positive integer sequence $\{q(b,\,i)_j\}$  with $i\sspin \{\range{1}{B(b)}\}$ and $j\sspin \{\range{1}{pes(b}\}$ in the complete system $SBB(b)$ by {\sl Schick} \cite{Schick} and {\sl Br\"andli}-{\sl Beyne} \cite{BB} from section {\bf B) i)} and {\it ii)} of the {\it Introduction}, periodic doubling sequences are used, namely $\{\modstaran{a(b,\,i)\,2^j}{b}\}$, with certain odd numbers $a(b,\,i)$ coprime to $b$, where $i\sspin \{\range{1}{c^*(b)}\}$ and $j\sspin \{\range{1}{P(b)}\}$.\pn
It will turn out that $c^*(b) \sspeq c(b)\sspeq B(b)$, and the length of the period $P(b)$, independent of $i$, is identical with $k(b)\sspeq pes(b)$.\pn
The short abbreviation $MDS$ is used instead of $MMDS$, for Modified Modular Doubling Sequence. 
\begin{definition} \label{MDS}
\pn
\begin{enumerate}
\item A doubling sequence $DSseq(b,\,i)$, for $i\sspin\{\range{1}{c^*(b)}\}$, is defined by \pn
$\{mod(a(b,\,i)\,2^j,\,b)\}_{j\sspgeq 1}$, for $b\sspeq 2\,n\sspp 1$, with $n\sspin \N$, and some odd element $a(b,\,i)\sspin RRS^*(b)$.\psn
\item A modified modular doubling sequence $MDSseq(b,\,i)$, for $i\sspin\{\range{1}{c^*(b)}\}$, is defined by $\{\modstaran{a(b,\,i)\,2^j}{b}\}_{j\sspgeq 1}$, for $b\sspeq 2\,n\sspp 1$, with $n\sspin \N$, and some odd element $a(b,\,i)\sspin RRS^*(b)$.
\end{enumerate}
\end{definition}
\pn
Note that the input $a(b,i)$ appears not as first number (for $j\sspeq 1$) of these sequences. In the $MDSseq(b,\,i)$ case it will be shown to appear at the end of the first period (as it would also in the other sequence with a doubled length of the period). \psn
The odd members of $RRS^*(b)$ are denoted by $RRS^*odd(b)$.  $\#RRS^*odd(b)\sspeq$\seqnum{A332435}$((b-1)/2)$, and $\#RRS^*even(b)\sspeq$\seqnum{A332436}$((b-1)/2)$. 
\pn
\begin{example} \label{Ex7}
 $RRS^*odd(33)\sspeq \{1,\, 5,\, 7,\, 13\}$,\, $MDSseq(33,\,1)\sspeq \{\text{repeat}(2,\,4,\,8,\, 16,\,1)\}$,\pn
 $MDSseq(33,\,2)\sspeq \{\text{repeat}(10,\,13,\,7,\, 14,\,5)\}$\,.
\end{example}
\pn
The two sequences of the example are periodic with the same length of period, \ie  $P(33,\, 1)\sspeq  P(33,\, 2)\sspeq 5$. 
\pn In general one starts with input $a(b,\, 1) \sspeq 1$, then if not all odd numbers from $RRS^*odd(b)$ are present in the period of the sequence one takes as next input, $a(b, 2)$, the smallest missing number in $RRS^*odd(b)$, in this example $5$, \etc, until all odd numbers from  $RRS^*odd(b)$ are present, and then the system $MDS(b)$ of $c^*(b)$ sequences is complete. In the example only two sequences are needed ($c^*(33)\sspeq 2$). This procedure is similar to the one of $SBB(b)$ systems with the different initial values for $q_0(b,\,i)$.\psn
After it has been proved that $MDSseq(b,\,i)$ is always periodic with the same period length for different $i$, hence called $P(b)$, the primitive period will be denoted by $MDS(b,\,i)$.
\begin{proposition} \label{MDSth}
{\bf Complete $\bf MDS(b)$ system}
\pn
\begin{enumerate}
\item Each sequence $MDSseq(b,\,i)$ is periodic, and the length of the primitive period  $MDS(b,\,i)$ is independent of any input $a(b,\,i)\sspin RRS^*odd(b)$, for $b\sspgeq 3$. This length of the period will be called $P(b)$.
\item $P(b)\sspeq k(b)\sspeq pes(b)\sspeq$\seqnum{A003558}$(b)$. 
\item $MDS(b,\,i)_{P(b)}\sspeq a(b,i)$.
\item \dstyle{\bigcup_{i=1}^{c^*(b)} MDS(b,\,i)\sspeq RRS^*(b)}.
\item $c^*(b)\,P(b) \sspeq \frac{\varphi(b)}{2}$. Hence $c^*(b)\sspeq c(b)\sspeq  B(b)\sspeq$\seqnum{A135303}$(b)$, for $b\sspgeq 3$. 
\end{enumerate}
\end{proposition}
\psn
\begin{proof} 
\pn
\begin{enumerate} 
\item The condition for the primitive length of the period $P(b,\,i)$ of $MDSseq(b,\,i)$, if existing, is $\modstaran{a(b,i)\,2^{P(b,\,i)\sspp 1}}{b}\sspeq  \modstaran{a(b,i)\,2^1}{b}$ for the smallest $P(b,\,i) \sspgr 0$. Because $0\sspkl a(b,\,i)\sspleq \floor{\frac{b-1}{2}}$, hence $\modan{a(b,\,i)}{b} \sspeq a(b\,\,i)$, it follows that $\modstaran{a(b,\,i)}{b}\sspeq a(b,\,i)$. Multiplicativity of $\text{mod}^*$ tells that the condition is $\modstaran{2^{P(b,\,i)}}{b}\sspeq 1$, with $1 \sspin RRS^*(b)$, representing the unique identity element of $\text {mod}^*$. Hence a solution exists, because there is always a solution for $\modan{2^{P(b,\,i)}}{b}\sspeq +1$, namely the $order(2,\,b)$. But this $P(b,\,i)$ may not give the primitive length of the period (see the next part). Therefore $MDSseq(b,\,i)$ is certainly periodic for each $i$, independently of $a(b,\,i)$.
\item  $\modstaran{2^{P(b)}}{b}\sspeq 1$ means that $2^{P(b)}\sspequiv \amodn{+1}{b}$ or $2^{P(b)}\sspequiv \amodn{-1}{b}$. For the primitive length of the period the least $P(b)\sspgr 0$ has to be chosen. This is the same congruence problem as the one for $k(b)$ of the coach system $\Sigma(b)$, and also for $pes(b)$ of the $SBB(b)$ system.
\item $MDS(b,\,i)_{P(b)}\sspeq \modstaran{a(b,\,i)\,2^{P(b)}}{b}$. Because $0\sspkl a(b,\,i) \sspleq \frac{b-1}{2}$, and if $\modan{2^{P(b)}}{b}\sspeq +1$ then $\text{mod}^*$ becomes $\text{mod}$, and the result is obviously $a(b,\,i)$. If $\modan{2^{P(b)}}{b}\sspeq -1$ then $\modan{-a(b,\,i)}{b} \sspgeq \frac{b-1}{2}$ and the result is $\modan{+a(i,\,b)}{b}\sspeq a(i,\,b)$\,.
\item By construction of $MDS(b)$ all odd numbers of $RRS^*$ appear once in the union of $MDS(b,\,i)$, for $i\sspeq \range{1}{c^*(b)}$. Each even number $2\,k$ of the complete set $MDS(b)$ has a smaller odd ancestor \seqnum{A000265}$(2\,k)$ (its odd part) in some cycle $MDS(b,\,i)$ with a certain odd initial value. Every even element of $RRS^*$, always $\sspleq \frac{b-1}{2}$, is in one of the $c^*(b)$ cycles of $MDS(b)$ because its odd ancestor, always $\sspkl  \frac{b-1}{2}$, has to be in precisely one cycle $MDS(b,i)$, with a certain odd initial value.
\item A corollary of {\it 1.}, {\it 4.} and {\it 2.} because \dstyle{\#RRS^*(b)\sspeq\sigma(b)\sspeq \frac{\varphi(b)}{2}}.\pn
$\sum_{i=1}^{c^*(b)}\,P(b,\,i) \sspeq c^*(b)\,P(b)\sspeq \#RRS^*(b)$. Then {\it 2.} together with eqs.~\ref{coachth} and ~\ref{SchickId} are used.
\end{enumerate}
\end{proof}
\psn
For the cycle system $MDS(b)$ for $b\sspeq 2\,n\sspp 1$, with $n\sspeq \range{1}{35}$, see {\it Table 2}, and also \seqnum{A334430}.
\pbn
$MDS(b,\,i)_j$, for $j\sspeq \range{1}{P(b)}$, are the elements of the cycle $MDS(b,\,i)$. For simplicity the argument $(b,\,i)$ will by suppressed, and instead of $MDS_j$ just $c_j$ ($c$ for cycle) is used in the following {\it Proposition}.  
\begin{proposition}\label{RecMDS}
\psn
For each cycle $MDS(b,\,i)$ the members $c_j\sspeq c(b,\,i)_j$ satisfy the recurrence
\Beq
c_j \sspeq \Cases2{2\,c_{j-1}}{$\text{if}\  \ c_{j-1}\sspleq \floor{\frac{b-1}{4}}$\,,}{b\sspm 2\,c_{j-1}}{$\text{if} \ \  c_{j-1}\sspgr \floor{\frac{b-1}{4}}$\,,}
\Eeq
for $j\sspgeq 2$, and the input is $c_1\sspeq \modstaran{2\,a(b,\,i)}{b}$.
\end{proposition} 
\begin{proof}
This uses the $\text{mod}^*$ definition of $c_j$, and multiplicativity of $\text{mod}$ to extract a factor 2. The second line follows from the $c_{j}$ definition in the case when $ (b-1)/2\sspkl \modan{2\,a\,2^{j-1}}{b}\sspleq (b-1)$, \ie\, $ (b-1)/4 \sspkl \modan{a\,2^{j-1}}{b}\sspleq \frac{b-1}{2}$. In this case the  $\text{mod}^*$ definition of $c_{j-1}$ tells that $c_{j-1}\sspeq \modan{a\,2^{j-1}}{b}$. Because in this case $c_j \sspeq \modan{-2\,a\,2^{j-1}}{b}\sspeq -2\,c_{j-1}$, that is $b\sspm 2\,c_{j-1}$ because the numbers $c$ are always from $RRS^*(b)$, hence $\sspleq \frac{b-1}{2}$. The simpler first line is proved in a similar manner.    
\end{proof}
\pn
This can be written as one line formulae.
\begin{corollary} \label{RecMDSCor}
\psn
\Beqarray
c_j &\sspeq& \frac{1}{2} \left( b\sspm \abs{b\sspm 4\,c_{j-1}}\right).\nonumber \\
c_j&\sspeq& \modstaran{2\,c_{j-1}}{b}\,. \nonumber \\
c_j&\sspeq& \modan{(-1)^{p_j}\,2\,c_{j-1}}{b},\  \  \text{where}\ \ p_j\sspeq
\text{parity}(c_j).
\Eeqarray
\end{corollary}
\psn
for $j\sspeq \range{2}{P(b)}$, and the initial value is $c_1 \sspeq c(b,\,i)_1$, that is $2\,a(b,\, i)$ if $a(b,\, i)\sspleq \floor{\frac{b-1}{4}}$, and $b\sspm 2\,a(b,\,i)$ if $a(b,\,i)\sspgr \floor{\frac{b-1}{4}}$. In the second equation $c(b,\,i)_{j-1}\sspin RRS^*(b,\,i)$ from {\it Proposition}~\ref{MDSth} {\it 4.} has been used, hence 
$\text {mod}(2\,c(b,i)_{j-1},\,b)\sspeq  2\,c(b,i)_{j-1}$. In the last equation the recurrence is considered $\text{mod}\, b$. In the upper alternative $c_j$ is even, and the lower one odd.   
\begin{example} \label{Ex7}
$b\sspeq 63,\, c^*(63)\sspeq 3,\, P(63)\sspeq 6$, $\floor{\frac{63-1}{4}}\sspeq 15$.
\pn $MDS(63,\,3)\sspeq [22,\,19,\,25,\,13,\,26,\,11]$. $c_2(63,\,3)\sspeq 19\sspgr 15$, $c_3(63,\,3)\sspeq 63$ $\sspm 2\cdot 19\sspeq \frac{1}{2}\,(63\sspm (4\cdot 19\sspm 63)) \sspeq 25$.
\end{example}
\pbn
Like in the complete system $SBB(b)$ the elements (numbers) of the cycles of $MDS(b)$ are interpreted as length ratios diagonal/radius in regular $(2\,b)$-gons as follows. A number $k$ (even or odd) from a cycle $MDS(b,i)$, hence $k\sspin [1,\,(b-1)/2]$, for $i\sspin \{\range{1}{c^*(b)}\}$, is mapped to \dstyle{2\,\cos\left(\frac{\pi\,k}{b}\right)\sspeq R(k,\,\rho(b))\sspgr 0}, where $R$ are monic {\sl Chebyshev} polynomials shown in \seqnum{A127672}, and \dstyle{\rho(b)\sspeq 2\,cos\left(\frac{\pi}{b}\right)}, the length ratio diagonal/side of the diagonal $d_2^{(b)}\sspeq \overline{V_0^{(b)}\,V_2^{(b)}}$ in the regular $b$-gon. [The $R$ polynomials are in \cite{ASt}, Table 22.7, p. 797,  called $C$. In \cite{KA}, Table 2, p. 72, they are called {\sl Lucas} polynomials.] \psn
The connection to length ratios diagonal/radius in the regular $(2\,b)$-gon comes from the trigonometric identity
\Beq
2\,\cos\left(\frac{\pi\,k}{b}\right) \sspeq 2\,\sin\left(\frac{\pi\,(b\ssppm 2\,k)}{2\,b}\right)\,.
\Eeq
For the positive sign see also the table and a comment in \seqnum{A082375}. But here the negative sign will be chosen. As mentioned above all values are positive. This is an advantage over considering cosine arguments $a(b,\,i)\,2^k$ like in the later discussed complete iteration system $IcoS(b)$.
\psn
\Eg\, $b\sspeq 5$, $MDS(5,\,1)\sspeq (2,\,1)$: $k\sspeq 1$ corresponds to the two diagonal/radius ratios $d_7^{(10)}$ and $d_3^{(10)}$ both of length  \dstyle{\varphi}, with the golden ratio $\varphi\sspeq 1.61803398...\,\sspeq$\seqnum{A001622} (in the lower and upper half-plane, respectively ). For $k\sspeq 2 $ these are $d_9^{(10)}$ and $d_1^{(10)}$, the side/radius ratios of length $\varphi\sspm 1$.
\psn
The numbers \dstyle{2\,\cos\left(\frac{\pi\,k}{b}\right)\sspeq R(k,\,\rho(b))}, with $k\sspin RRS^*(b)$, can be taken as zeros of a monic polynomial $P^*(b,x)$ of degree $\sigma(b)\sspeq$\seqnum{A023022}$(b)\sspeq$\seqnum{A055034}$(b)\sspeq \delta(b)$. It follows that this is a polynomial with coefficients from the ring of integers of the simple algebraic field extension $\Q(\rho(b))$, denoted by \dstyle{{\cal O}_{\Q(\rho(b))}\sspfed \Z[\rho(b)]}. The following proposition gives the formula. See {\it Table 1} for $b \sspeq 2\,n+1$. for $n\sspeq\range{1}{10}$.  
\psn
\begin{proposition} 
{\bf $\bf P^*(b,\,x)$ polynomials}
\pn
The monic polynomials 
\Beq
P^*(b, x)\sspdef \prod_{j=1}^{\sigma(b)} \left(x\sspm 2\,\cos\left(\frac{\pi\,RRS^*(b)_j }{b}\right)\right)  \sspeq \prod_{j=1}^{\sigma(b)} (x\sspm R(RRS^*(b)_j,\,\rho(b))\Big |_{C(b,\,\rho(b))\sspeq 0}\ .
\Eeq
with $b\sspeq 2\,n\sspp 1$, $n\sspin \N$, have coefficients which are integers in the simple field extension $\Q(\rho(b))$, where \dstyle{\rho(b)\sspeq 2\,\cos\left(\frac{\pi}{b}\right)},
\end{proposition}
\pn
In the evaluation the minimal polynomial $C(b,\, x)$ of $\rho(b)$, given in \seqnum{A187360}, is used to eliminate powers of $\rho(b)$ with exponents larger than $\delta(b) - 1\sspeq \sigma(b) - 1$.  
\psn
\begin{proof}
The first equation gives the definition. The last equation follows from the definition of the $R$ polynomials in terms of {\sl Chebyshev} $T$ polynomials, and their trigonometric definition. Then the coefficients will be expressed in the power basis of $\Q(\rho(b))$ of degree $[\Q(\rho(b)):\Q] \sspeq \delta(b)$ with rational integer coefficients, because $R\sspin \Z[x]$.
\end{proof}
\begin{example} \label{Ex8}
\Beqarray
RRS^*(9)&\sspeq&[1,\,2,\,4], \ \sigma(9)\sspeq 3, \nonumber\\
P^*(9,\,x) &\sspeq& (x\sspm 2\,\cos(\pi/9))\,(x\sspm 2\,\cos(\pi\,2/9))\,(x\sspm 2\,\,\cos(\pi\,4/9)), \nonumber \\
&\sspeq& x^3 \sspm 2\,\rho(9)\,x^2 \sspp (-3 \sspp 2\,\rho(9)^2)\,x \sspm 1, \nonumber \\
&& \text{with}\ \rho(9)\sspeq 2\,\cos(\pi/9)\sspeq\seqnum{A332437}\sspeq 1.87938524...\,.
\Eeqarray
\end{example}
\psn
{\bf C)}\ \ {\bf Complete $\bf IcoS$-system}
\psn
An alternative approach to the modified modular doubling sequence has been considered by {\sl Gary W. Adamson} and {J. Kappraff} \cite{KA} by studying iteration of the quadratic polynomial $R(2,\,x)\sspeq x^2 \sspm 2$  with certain seeds (see above for {\sl Chebyshev} $R$ polynomials, and the name {\sl Lucas} polynomial in \cite{KA}). It has also been used as {\it r-t} table ({\it r-t} for root trajectory) by {\sl Gary W. Adamson} in his Aug 25 2019 comment in \seqnum{A065941}. As will be shown this is very similar to the $MDS$-system. The iteration of  $x^2 \sspm 2$ is proposed in a procedure 1. by {\sl Devaney} \cite{Devaney}, on p. 25, for investigation, and in chapter 7.1, pp. 69 - 71, this function is called $Q_{-2}(x)$. It is shown in chapter 10.2, pp. 121 - 126 to be semiconjugate to the function $V(x)\sspeq 2\abs{x} \sspm 2$, and also to the angle doubling function $D(\theta)\sspeq 2\,\theta $ on the unit circle.
\psn 
If one starts the iteration of $R(2,\,x)$ with a seed $2\,\cos(\pi\,1/b)$, for $b\sspeq 2\,n\sspp 1$, with $n\sspin \N$, one does not obtain a purely periodic sequence if signs are respected. In order to obtain purely periodic sequences one starts the iteration with $R^{[1]}(2,\,x)$, not with the identity map $R^{[0]}(2,\,x)\sspeq x$. The later given theorem shows periodicity with primitive period length $P(b)$ like for $MDS(b)$, and also the number of seeds $c^*(b)$ necessary to obtain finally all $2\,\cos(\pi\,j/b)$, values with $j\sspin RRS^*(b)$.
Note that not all iterations will be positive. Hence for the interpretation as length ratios diagonal/radius in $(2\,b)$-gons the signs will have to be ignored. \psn
This complete system of periodic signed cosine sequences obtained by iteration will be called  $IcoS(b)$, and its elements are denoted by sets $IcoS(b,\,i)$, for $i\sspeq \range{1}{c^*(b)}$.
\psn
\begin{lemma} \label{RpolPeriodicity}
The iteration with the {\sl Chebyshev} polynomial $R(2,\, x)\sspeq x^2\sspm 2$ is periodic  for any seed $x\sspeq 2\,\cos(\pi\,j/b)\sspeq R(j,\,\rho(b))$, where $\rho(b)\sspeq 2\,\cos(\pi/b)$, and odd $j\sspin RRS^*(b)$, and the primitive length of the period is $P(b)\sspeq$\seqnum{A003558}$((b-1)/2)$. 
\end{lemma}
\begin{proof}
The standard formula for the {\sl Chebyshev} $T-$polynomials (see \seqnum{A053120}, and \cite{Rivlin}, p. 5, Exercise 1.1.6) $T(n,\, T(m,\,x)) \sspeq T(n\,m,\, x)$ translates for $R(n, x)\sspeq 2\,T(n,\,x/2)$ to $R(n,\,R(m,\,x)) \sspeq R(n\,m,\, x)$. The $T$ and $R$ polynomials with negative index are identical with the ones for positive index: $R(-n,\,x)\sspeq R(n,\,x)$, for $n\sspin \N_0$.\pn
From the periodicity of the cosine function, and $ R(j,\,\rho(n)) \sspeq 2\,\cos(\pi\,j/n)$, follows that $R(k,\, R(j,\, \rho(n))\sspeq R(k\,j,\, \rho(n)) \sspeq  R(k\,j\,\sspm q\,2\,n,\, \rho(n))$, for $n\sspin \N$ and $q\sspin \Z$. Together with the symmetry under sign flip of the index, periodicity requires that
\Beq
R(k\,j,\,\rho(n)) \sspeq R(\pm(k\,j \sspm q\,2\,n),\, \rho(n)), \ \text{for}\ \ n\sspin \N, \ \text{and}\  q\sspin \Z.
\Eeq
For the iteration of $R^{[k]}(2,\,x)\sspeq R(2^k,\,x)$, for $k\sspin \N$, periodicity therefore means in the present context, that, with any seed $x\sspeq 2\,\cos(\pi\,j/b)$, and positive odd $j$ with $\gcd(j,\,b)\sspeq 1$ and $j\sspleq \frac{b-1}{2}$, \ie $j\sspin RRS^*odd(b)$, 
\Beq
R(2^k\,j,\,\rho(b)) \sspeq R(\pm(2^k\,j \sspm q\,2\,b),\,\rho(b)) \shouldbe R(j\,2^1,\,\rho(b)),
\Eeq
For the smallest $k\sspgr 1$, and then the primitive length of the period is $P(b) \sspeq k-1$. 
This leads to $q\,b\sspeq j\,(2^{k-1}\mp 1)\sspgr 0$, \ie \ $q\sspeq q^{\prime}\,j$, with positive integer $q^{\prime}$ because $j \notdiv b$. Therefore, $j$ drops out and $2^{k-1} \sspequiv \amodn{\pm 1}{b}$. Thus, $P(b)\sspeq pes(b)\sspeq$\seqnum{A003558}$((b-1)/2)$.
\end{proof}
\psn
In order to see the close relationship to the $MDS$-system the following {\it lemma} will establish the connection to the $\text{mod}^{*}$ congruence.\psn
\begin{lemma}
\Beq
R(a(b,\,i)\,2^k,\,\rho(b))\sspeq 2\,\cos\left(\frac{\pi\,a(b,\,i)\,2^k}{b}\right)\sspeq \pm\,2\,\cos\left(\frac{\pi\,\widehat j}{b} \right)\, \ \  \text {with} \ \ \widehat j\sspeq \modstaran{a(b,\,i)\,2^k} {b}\,,  
\Eeq
and the sign is $+$ for $x\sspdef \modan{a(b,\,i)\,2^k}{2\,b}\sspin I\sspdef (0,\,\frac{b}{2})$ or $x\sspin IV\sspdef [\frac{b}{2},\,b)$, and the sign is $-$ if $x\sspin II\sspdef [b,\,\frac{3\,b}{2})$ or $x\sspin III\sspdef [\frac{3\,b}{2},\,2b-1)$. 
\end{lemma} 
\pn
The identification of the unsigned $2\,\cos(\frac{\pi\,\widehat j}{b})$ values with length ratios diagonal/radius in a regular $(2\,b)$-gon are then obtained with the help of \Eq{15}.
\psn
\begin{proof}
\pn
First note that the first term ($k\sspeq 1 $) is only positive if $0\sspkl a(b,\,i) \sspleq \frac{b-1}{4}$. The only negative case for $a(b,\,1)\sspeq 1$ is $IcoS(3,1)\sspeq \{-1\}$. Larger seeds are only needed for $b\sspeq 17,\, 31,\,33,\,... $ (\seqnum{A333855}), and seem to lead always to positive values of the first term. \Eg\, even for $b\sspeq 127$ with $P(127)\sspeq 7$ and $c^*(127)\sspeq 9$ the largest seed is $21\sspkl 31\sspeq \floor{\frac{b-1}{4}}$. 
\psn
It is always possible to simplify to $\pm\,2\,\cos(\frac{\pi\,\widehat j}{b})$ based on the periodicity of cosine, which implies that only $\modan{a(b,\,i)\,2^k}{2\,b}$ enters. Consider the four intervals $I$ to $IV$, given in the {\it lemma}, with signs 
$+,\,-,\,-,\,+$\, of $2\,\cos(\pi\,x/b)$, respectively. Here $x \sspeq \modan{a(b,i)\,2^k}{2\,b}$ is a positive integer $\sspin [1,\,2\,(b-1)]$. If $x_1\sspin I$ then $x_1\sspeq \modan{x_1}{b} \sspeq \modstaran{x_1}{b}$ because $x_1\sspleq \floor{\frac{b}{2}} \sspkl \frac{b}{2}$, and its cosine value is positive. If $x_4\sspin IV$ one has $\modan{x_4}{b}\sspeq x_4\sspm b \sspin II$. Hence $\modstaran{x_4}{b}\sspeq \modan{-x_4}{b}\sspeq -x_4\sspp 2\,b\sspin I$. Similarly for $x_2\sspin II$ and for $x_3\sspin III$, both with negative cosine values, leading to $\modstaran{x_2}{b}\sspeq -x_2\sspp b\sspin I$, and $\modstaran{x_3}{b}\sspeq -x_3\sspp b\sspin I$, respectively.\pn 
\end{proof}
\psn
\begin{definition} \label{IcoS}
{\bf Complete system $\bf IcoS$}\pn
The iteration procedure $\{R^{[k]}(2,\,x)\}_{k\sspgeq 1}$ uses first the seed $x\sspeq 2\,\cos(\frac{\pi}{b})\sspeq \rho(b)$. If in the period $ IcoS(b,\,1)\sspeq \{R(1\cdot 2^k,\,\rho(b))\}_{k\sspgeq 1}^{P(b)}$, consisting after trigonometric simplifications of terms $\pm\,2\,\cos\left(\frac{\pi\,\widehat j}{b}\right)$ (shown in the preceding {\it lemma}), all odd and even numbers $\widehat j\sspin RRS^*(b)$ appear then only this seed $a(b,\,1)$ is needed, and $IcoS(b)\sspeq IcoS(b,\,1)$ is complete. Otherwise another iteration is started using a seed $2\,\cos(\pi\,a(b,\,2)/b)$ with the smallest odd member $a(b, 2)\sspin RRS^*(b)$ that did not appear in $IcoS(b,\,1)$. This gives the period $IcoS(b,\,2)\sspeq \{R(a(b,\,2)\cdot 2^k,\,\rho(b))\}_{k\sspgeq 1}^{P(b)}$, etc., and the $IcoS(b)$ system is complete after all numbers of  $RRS^*(b)$ have been reached.
\end{definition}
\psn
\begin{example} \label{Ex9}
\Beqarray
b\sspeq 17,\,P(b)&\sspeq& 4,\  c^*(b)\sspeq 2: \nonumber\\
IcoS(17,\,1)&\sspeq& \{2\,\cos(\pi\,(2/17)),\, 2\,\cos(\pi\,(4/17)),\, 2\,\cos(\pi\,(8/17)), -2\,\cos(\pi\,(1/17)) \} \nonumber \\
&\sspeq&\{ 1.8649...,\, 1.4780...,\, .1845...,\, -1.9659...\}.\nonumber \\
IcoS(17,\,2)&\sspeq& \{2\,\cos(\pi\,(6/17)),\, -2\,\cos(\pi\,(5/17)),\, -2\,\cos(\pi\,(7/17)),\, -2\,\cos(\pi\,(3/17))\} \nonumber\\
&\sspeq& \{.8914...,\, -1.2052...,\, -.5473...\, -1.7004...\}.
\Eeqarray
\end{example} 
\begin{theorem}
The complete system $IcoS(b)$, with $b\sspeq 2\,n\sspp 1$, consist of \dstyle{c(b)\sspeq \frac{\varphi(b)}{2\,P(b)}\sspeq}\seqnum{A135303}$(n)$ disjoint periods $IcoS(b,\,i)$ of length $P(b)\sspeq $\seqnum{A003558}$(n)$, and each number $\widehat j\sspin RRS^*(b)$ appears once in the elements $\pm\, 2\,\cos\left(\frac{\pi\,\widehat j}{b}\right)$ of this set of complete periods.
\end{theorem}
\begin{proof}
By construction of $IcoS(b)$ (definition ~\ref{IcoS}) every odd number $\widehat j$ of $RRS^*(b)$ is in one of the periods $IcoS(B,\,i)$.
The proof for the even numbers $\widehat j$ can be taken over, {\it mutatis mutandis}, from {\it Theorem} ~\ref{MDSth}, {\it 4} and {\it 5}.
\end{proof}
\pn
This shows that $IcoS(b)$ is very similar to the $MDS(b)$ (see the recurrence ~\ref{RecMDSCor}), but works with signed cosine values. The signs have to be ignored for the interpretation as diagonal/radius lengths ratios in the regular $(2\,b)$-gon. Therefore, the $MDS(b)$ system is preferred because it works with positive integers and maps to positive cosine and sine values.
\section{Equivalence of the three complete systems}
\psn
{\bf A1) $\bf MDS(b) \Longrightarrow \Sigma(b)$}
\psn
The proof will be given element-wise:  $MDS(b,\,i) \Longrightarrow \Sigma(b,\,i)$, for odd $b\sspgeq 3$ and $i \sspeq \range{1}{c^*(b)}$. The elements of $MDS(b,\,i)$ are $c(b,\,i)_j$, abbreviated as $c_j$, for $j\sspeq \range{1}{P(b)}$. It has been shown in {\it Theorem} ~\ref{MDSth} {\it 5.} that $c^*(b)\sspeq c(b)$, and in ~\ref{MDSth} {\it 2.} that $P(b)\sspeq k(b)$. 
\psn
For the proof the recurrence relation of the odd members of $MDS(b,\,i)$ will be needed. For this one considers the cycle  $MDS^{\prime}(b,\,i)$ with the input $a(b,\,i)$ as first element $c^{\prime}_0$, and $c^{\prime}_j\sspeq c_j$, for $j\sspeq \range{1}{P(b)-1}$. This $MDS^{\prime}$ appears (in a different notation) in \cite{HP}, at the top of p. 103, in the proof of the quasi-order theorem.  
\psn
\begin{lemma} {\bf Recurrence for odd elements of $\bf MDS^{\prime}(b,\,i)$}
\psn
The odd elements of  $MDS^{\prime}(b,\,i)$, collected in the list $Co^{\prime}(b,\,i)$ with elements ${co^{\prime}}(b,\,i)_j\sspeq co^{\prime}_j$, for $j \sspeq \range{0}{r^*(b,\,i)-1}$, satisfy the recurrence relation (with $l^{\prime}_j\sspeq l^{\prime}(b,\,i)_j$)
\Beq
co^{\prime}_j\sspeq b\sspm 2^{{l^{\prime}_j}}\,co^{\prime}_{j-1}, \ \ \text{for} \  \ j\sspeq \range{1}{r^*(b,\,i)-1},
\Eeq
with input $co^{\prime}_0\sspeq a(b,\,i)$, and  $i\sspin \{\range{1}{c^*(b)}\}$. The $l^{\prime}$-tuple of length $r^*(b,\,i)$ gives the differences of the indices $ind$ of consecutive odd numbers in $MDS^{\prime}(b,\,i)$. Its elements are  
\Beq
l^{\prime}_j\sspeq ind(co^{\prime}_j) \sspm ind(co^{\prime}_{j-1}),\ \ \text{for} \ \ j\sspeq \range{1}{r^*(b,\,i)},
\Eeq
with $ind(co^{\prime}_{r^*(b,\,i)}) \sspeq P(b)$.
\end{lemma}
\begin{proof}
Between the two consecutive odd numbers $co^{\prime}_{j-1}\sspeq c^{\prime}_{\sum_{i=1}^{j-1}\,l_{i}^{\prime}}$ and $co^{\prime}_j\sspeq c^{\prime}_{\sum_{i=1}^{j-1}\,l_{i}^{\prime} + l_j^{\prime}}$ (empty sums vanish) there are $l_j^{\prime}\sspm 1$ doublings in $MDS^{\prime}$, and for the next odd number the second version of the recurrence of the $c^{\prime}\sspeq c$ elements given in \Eq{15} of {\it Proposition} \ref{RecMDS}, applies. This means that $co^{\prime}_{j}\sspeq b \sspm 2\,(2^{l^{\prime}_j-1})\,co^{\prime}_{j-1}\sspeq b\sspm 2^{l^{\prime}_j}\,co^{\prime}_{j-1}$.  
\end{proof}
\psn
For the odd elements of of $MDS(b,\,i)$, given by $Co(b,\,i)\sspeq [co(b,\,i)_1,\,...,\, co(b,\,i)_{r^*(b,\,i)}]$, this translates to
\begin{corollary} {\bf Recurrence for odd elements of $\bf MDS(b,\,i)$} \label{MDSoddCor}
\Beqarray
&&co(b,\,i)_j\sspeq b\sspm 2^{l^{\prime}(b,\,i)_j}\,co(b.\,i)_{j-1}, \ \ \text{for}\ \ j\sspeq \range{2}{r^*(b,\,i)}, \nonumber  \\
&&\text{and} \ \  co(b,\,i)_1\sspeq b\sspm 2^{l^{\prime}(b,\,i)_1}\, a(b,\,i)\, .
\Eeqarray
\end{corollary}
\psn
\begin{example}\label{Ex9}
$b\sspeq 63$, $i\sspeq 2$: $P(63)\sspeq 6$, $r^*\sspeq 4$, $MDS\sspeq [10,\,20,\,23,\,17,\,29,\,5]$,\ $MDS^{\prime}\sspeq [5,\,10,\,20,\,23,\,17,\,29]$, $Co^{\prime}\sspeq [5,\,23,\,17,\,29]$, $l^{\prime}\sspeq [3,\,1,\,1,\,1]$, because \eg, $l^{\prime}_1\sspeq ind(23)\sspm ind(5) = 3\sspm 0 = 3$, and $l^{\prime}_4\sspeq P(63)\sspm ind(29)\sspeq 6\sspm 5\sspeq 1$. \pn
Recurrences: $17\sspeq co^{\prime}_2 \sspeq 63 \sspm 2^1\,co^{\prime}_1\sspeq 63\sspm 2\cdot 23\sspeq 17$. \pn
\hskip 2.5cm $29\sspeq co_3 \sspeq 63 \sspm 2^1\,co_2\sspeq 63\sspm 2\cdot 17\sspeq 29$.
\end{example}
\begin{definition} {\bf two lists ${\bf U(b,\,i)}$ and ${\bf L(b,\,i)}$}
\psn
Let $MBS(b,\,i)$, for odd integer $b \sspgeq 3$, and $i\sspin \{\range{1}{c^*(b)\sspeq c(b)}\}$, be the period of length $P(b)\sspeq k(b)$. The elements of the list $MBS(b,\,i)$ are denoted by $c(b,\,i)_j$ (abbreviated as $c_j$). Let the sub-list of odd numbers of $MBS(b,\,i)$ be denoted by $Co(b,\,i)$ of length $r^*(b,\,i)$ with elements $co_j$ (no interchange of the odd elements is allowed). The position (index) of $co_j$ in the list $MBS(b,\,i)$ is denoted by $ind(co_j)$.\pn 
\begin{enumerate}
\item The list $U(b,\,i)$ has elements $u_j$, for $j\sspeq \{\range{1}{r^*(b,\,i)}\}$, given by 
\Beq
u_j\sspeq co_{r^*(b,\,i)-j+1}\,,
\Eeq
\ie the odd numbers of $MBS(b,\,i)$ are read backwards. 
\item The list $L(b,\,i)$ has elements $l_j$, for $j\sspeq \{\range{1}{r^*(b,\,i)}\}$, given by 
\Beq
l_j\sspeq ind(u_j) \sspm ind(u_{j+1}),\ \  \text{with}\  \ ind(u_{r^*(b,\,i) +1}) \sspdef 0,
\Eeq
\ie the number of steps to go in $ MDS(b,\,i)$ backwards from one odd number to the next is counted (in a cyclic manner) .
\end{enumerate} 
\end{definition}
\begin{example}\label{Ex10}
$b\sspeq 63$, $c^*(63)\sspeq 3$, $P(63) \sspeq 6$.\pn
$MDS(63,3)\sspeq [22,\,19,\,25,\,13,\,26,\,11]$, $Co(63,\,3)\sspeq [19,\,25,\,13,\,11]$, $r^*(63,\,3)\sspeq 4$, \pn
$ind(19)\sspeq 2$, $ind(25)\sspeq 3$, $ind(13)\sspeq 4$, $ind(11)\sspeq P(63)\sspeq 6$.\pn
$U(63,\,3)\sspeq [11,\,13,\,25,\,19]$, $\text{length}(U(63,\,3))\sspeq r(63,\,3)\sspeq 4$.\pn
$L(63,\,3)\sspeq [2,\,1,\,1,\,2]$, $\text{length}(L(63,\,3))\sspeq \text{length}(U(63,\,3))$\,.\pn
\end{example}
\begin{proposition} \label{MDSSigma}
Let $MBS(b,\,i)$, $U(b,\,i)$ and $L(b,\,i)$ be as in the definition above, then 
\Beq
 U(b,\,i) \sspeq A(b,\,i),\ \ \text{and}\ \ U(b,\,i) \sspeq K(b,\,i),\   
\Eeq
with the upper and lower lines $A(b,\,i)$ and $K(b,\,i)$ of the coach $\Sigma(b,\,i)$, respectively.
\end{proposition} 
\begin{proof}
By definition of $U(b,\,i)$ the length $r^*(b,\,i)$ has been identified with $r(b,\, i)$, the length of coach $\Sigma(b,\,i)$. This is possible because the recurrence of the $MDS(b,\,i)$ elements $c_j$, when considered modulo $b$, satisfy the third equation of the {\it Corollary} \ref{RecMDSCor}. This implies, with {\it Proposition} \ref{MDSth}, {\it 3}, that $ a(b,\,i) \sspeq c_{P(b)} \sspeq \modan{a(b,\,i)\,(-1)^{r^*(b,\,i)}\, 2^{P(b)}}{b}$ because the number of sign flips starting with $c_1$, ending with $c_{P(b)}$ is the number $r^*(b,\,i)$ of odd numbers of $MDS(b,\,i)$. The input $a(b,\,i)\sspkl (b-1)/2$ and coprime to $b$ drops out, and from $P(b)\sspeq k(b)$ and the quasi-order theorem \ref{quoth} one can identify $r^*(b,\,i)$ with $r(b,\,i)$. 
\psn
From $U(b,\,i)$ and $L(b,\,i)$ the recurrence for coaches, \Eq{1}, can be proved by identifying $l(b,\,i)_j$ with $k(b,\,i)_j$. The reversed $l-$tuple 
is $l^{\prime}$, with $l^{\prime}(b,\,i)_j \sspeq l(b,\,i)_{r^*(b,\,i)-j+1}$.\pn
The input of $MDS(b,\,i)$ is $A(b,\,i)_1 = a(b,\,i) = u(b,\,i)_1 \sspeq co(b,\,i)_{r^*(b,\,i)}$. With \Eq{25} one has to prove 
\Beq
b\sspm co(b,\,i)_{r^*(b,\,i)-j+1} \sspeq 2^{l(b,\,i)_j}\,co(b,\,i)_{r^*(b,\,i)-j}, \ \ \text{for}\ \  j\sspeq \range{1}{r^*(b,\,i)-1}\, .
\Eeq
The input is now $co(b,\,i)_{r^*(b,\,i)}\sspeq a(b,\,i)$. An index change $\bar j\sspeq r^*(b,\,i)-j+1$, and the relation between $l$ and $l^{\prime}$ gives
\Beq
co(b,\,i)_{\bar j}\sspeq b\sspm 2^{l^{\prime}(b,\,i)_{\bar j}}\, co(b,\,i)_{\bar j\sspm 1},\ \ \text{for} \ \  \bar j \sspeq \range{2}{r^*(b,\,i)}\,,
\Eeq
and the input is now $co(b,\,i)_1\sspeq b\sspm 2^{l^{\prime}(b,\,i)_1}\, a(b,\,i)$.
But this is the proved recurrence from {\it Corollary} ~\ref{MDSoddCor} \Eq{26}.
\end{proof}
\pn
{\bf A2) $\bf \Sigma(b) \Longrightarrow MDS(b)$}
\psn
The proof for this direction has essentially been given by {\sl Hilton} and {\sl Pedersen} \cite{HP}, pp. 102 - 103, in their proof of the quasi-order theorem. It will be recapitulated here in a different notation. 
\begin{definition} {\rm \cite{HP}} {\bf Modified Coach Symbol $\bf MCSy(b,\,i)$}
\psn
For fixed odd $b\sspgeq 3$ and $i\sspin\{\range{1}{c(b)}\}$ the modified coach symbol $MCSy(b,\,i)$ has upper line $Cy(b,\,i)$ with odd numbers (in the following the argument $(b,\,i)$ will be mostly suppressed)
\Beq
cy_j\sspeq a_{r+2-j},\  \text{for} \  j\sspeq 1,\,2,\,...,\, r+1,
\Eeq
with the upper line coach numbers $a_j$ from $A(b,\,i)$, of length $r\sspeq r(b,\,i)$, and $a_{r+1}\sspeq a_1$, hence $cy_{1}\sspeq a_1\sspeq cy_{r+1}$.\pn
The lower line $Ly(b,\,i)$ of length $r$ has elements
\Beq
ly_j\sspeq k_{r+1-j}, \  \text{for} \  j\sspeq 1,\,2,\,...,\, r,
\Eeq 
with the lower line coach numbers $k_j\sspeq k(b,\,i)_j$ from $K(b,\,i)$.
\end{definition}
\pn
\Beq
\sum_{j=1}^{r}ly_j\sspeq \sum_{j=1}^{r}k_j\sspeq k(b),\ \  \text{for\ each}\  i.
\Eeq
\begin{definition} {\rm \cite{HP}} {\bf Extended Modified Coach Symbol $\bf EMCSy(b,\,i)$}
\psn
Between each $cy_j$ and $cy_{j+1}$ of the upper line $Cy(b,\,i)$ of $MCSy(b,\,i)$\  $ly_{j}-1$ consecutive doublings of $cy_j$ are inserted. This results in a one line list $EMCSy(b,\,i)$ of length $k(b)\sspp 1$ but now with offset index $0$.\pn
With $s_j\sspdef \sum_{q=1}^{j}\, ly_q$ (where $s_0\sspeq 0$), the elements $n_q$ of $EMCSy(b,\,i)$, for $q\sspeq \range{0}{k(b)\sspm 1}$, are formally
\Beq
n_{s_j\sspp p(j)} \sspeq 2^{p(j)}\,cy_{j+1},\ \ \text{for} \ \ p(j)\sspeq 0,\,1,\,...,\,ly_{j+1} - 1, \ \ \text{and} \ \ n_{k(b)\sspp 1}\sspeq cy_1,
\Eeq
for $j\sspin\{\range{0}{r(b,\,i)\sspm 1}\}$.
\end{definition}
\begin{example} \label{Ex11} 
$b\sspeq 63,\, i\sspeq 2,\, k(63)\sspeq 6,\, r(63,\,2)\sspeq 4$.\pn
$A\sspeq [5,\,29,\,17,\,23]$,\, $K\sspeq [1,\,1,\,1,\,3]$;\, $Cy\sspeq [5,\,23,\,17,\,29,\,5]$,\ $Ly \sspeq [3,\,1,\,1,\,1]$.\pn
$EMCSy(63,\,2)\sspeq [5,\,2^1\cdot5,\,2^2\cdot 5,\,23,\,17,\,29,\,5]\sspeq [5,\,10,\,20,\,23,\,17,\,29,\,5]$.
\end{example}
\pn
This shows, with the {\it Example} \ref{Ex9}, that $EMCSy(63,\,2)$ without its last entry coincides with $MDS^{\prime}(63,\,2)$, and without its first entry it becomes $MDS(63,\,2)$. 
\psn
\begin{proposition} {\bf $\bf MDS(b,\,i)$ from $\bf EMCSy(b,\,i)$ }
\psn
For odd $b\sspgeq 3$, and $i\sspin \{\range{1}{c(b)}\}$, the extended modified list $EMCSy(b,\,i)$, derived from the coach symbol $\Sigma(b,\,i)$,  gives without its first term $n_0\sspeq a(b,\,i)$ the modified modular doubling sequence $MDS(b,\,i)$.  
\end{proposition}
\begin{proof}
The recurrence relation with input $a(b,\,i)$ for $EMCSy(b,\,i)$ given in \cite{HP}. p. 103, \Eq{7.13}, for the entries $n_q$  (here with offset $q\sspeq 0$), coincides with the     
one given in {\it Corollary} \ref{RecMDSCor}, \Eq{16} for the $MDS(b,\,i)$ entries $c_j$.
\end{proof}
\pbn
{\bf B1) $\bf MDS(b) \Longrightarrow SBB(b)$}
\psn
The proof will be given element-wise:  $MDS(b,\,i) \Longrightarrow SBB(b,\,i)$, for odd $b\sspgeq 3$ and $i \sspin \{\range{1}{c^*(b)}\}$. It has been shown in {\it Proposition} ~\ref{MDSth} {\it 5.} that $c^*(b) \sspeq B(b)$, and in  ~\ref{MDSth} {\it 2.} that $P(b)\sspeq pes(b)$. The fixed arguments $(b,i)$ will be mostly suppressed.\psn
The $SBB$(b,\,i) recurrence for elements $q_j$ from \Eq{8} with input $q_0\sspeq a(b,\,i)$, an odd integer from $RRS^*(b)$, will be proved to follow from the recurrence {\it Proposition} ~\ref{RecMDS}, \Eq{15}, for the $MDS(b,\,i)$ cycle elements $c_j$ with input $a(b,i) = c_{P(b)}$ from {\it Proposition} ~\ref{MDSth} {\it 3.}\,.\psn
\begin{proposition}\ {\bf $\bf SBB(b,\,i)$ from $\bf MDS(b,\,i)$} \label{MDSSBB}
\psn
The elements of the (positive) $SBB(b,\i)$ cycle are given by
\Beq
q_j\sspeq b\sspm 2\,c_{P(b)-1+j}, \ \  \text{for} \ \ j\sspeq  
\range{0}{pes(b)},
\Eeq
where $c_j$ are the elements of the cycle $MDS(b,\,i)$ of length $P(b)\sspeq pes(b)$. $c_{P(b)+l}\sspeq c_l$, for $l\sspin \N$ (cyclicity). Hence the input is $q_0(b,\,i) \sspeq c(b,\,i)_{P(b)}\sspeq a(b,\,i)$.
\end{proposition}
\begin{example} \label{Ex12}
$b\sspeq 63$, $i\sspeq 2$, $P \sspeq 6 \sspeq pes$, $MDS(63,\,2) \sspeq [10,\,20,\,23,\,17,\,29,\,5]$: \pn
$q_0\sspeq  \sspeq 63\sspm 2\,c_5\sspeq c_6 \sspeq 5$, $q_1\sspeq 63\sspm 2\,c_6 \sspeq 53$, $q_2\sspeq 63\sspm 2\,c_7 = 63\sspm 2\,c_1 \sspeq 43$, $q_3\sspeq 63\sspm 2\,c_2 \sspeq 23$, $q_4\sspeq 63\sspm 2\,c_3 \sspeq 17$, $q_5\sspeq 63\sspm 2\,c_4 \sspeq 29$, hence $SBB(63,\, 2)\sspeq (5,\,53,\,43,\,23,\,17,\,29)$.   
\end{example}
\begin{proof}
The recurrences fit because   $2\,c_j\sspeq b\sspm \abs{b\sspm 4\,c_{j-1}}$  from {\it Corollary} ~\ref{RecMDSCor}. Thus
\Beqarray
q_j&\sspeq& b\sspm (b\sspm \abs{b\sspm 4\, c_{P(b)\sspm 2\sspp j}})\sspeq \abs{b\sspm 4\,c_{P(b)\sspm 2\sspp j}} \nonumber \\
&\sspeq& \abs{b\sspm 2\,(b\sspm 2\,c_{P(b)\sspm 2\sspp j})}\sspeq \abs{b\sspm 2\,q_{j-1}}\, .
\Eeqarray
The input $a(b,\,i)$ relation holds because of the recurrence
\Eq{15}, the lower line, because $c_{P(b)}\sspeq a(b,\,i)$ is odd, hence $q_0\sspeq b\sspm 2\,c_{P(b)\sspm 1} \sspeq c_{P(b)}\sspeq a(b,\,i)$\, .
\end{proof}
\pbn
{\bf B2) $\bf SBB(b) \Longrightarrow MDS(b)$}
\psn
This direction is clear from reading {\bf B1)} backwards.
One can use cyclicity of $SBB(b,\,i)$ and replace $q_{j+1-pes(b)}$ by $q_{j+1}$, and use $q_{pes(b)}\sspeq q_0$ and $q_{pes(b)+1}\sspeq q_1$.\psn
\begin{proposition}\ {\bf $\bf MDS(b,\,i)$ from $\bf SBB(b,\,i)$} \label{SBBMDS}
\psn
The elements of the cycle $MDS(b,\, i)$ are given by
\Beq
c_j\sspeq \frac{1}{2}\,(b\sspm q_{j+1}), \ \  \text{for} \ \ j\sspeq  \range{1}{P(b)},
\Eeq
where $q_j$ are the elements of the cycle $SBB(b,\,i)$ of length $pes(b)\sspeq P(b)$. $q_{pes(b) + l}\sspeq q_l$, for $l\sspin \N_0$ (cyclicity). Hence the input is $a(b,\,i)\sspeq c(b,\,i)_{P(b)} \sspeq q_0(b,\,i)$.
\end{proposition}
\begin{proof}
Use in {\bf B1)} the index change $j^{\prime} \sspeq P(b)\sspm 1 \sspp j$, cyclicity of $SBB(b,\,i)$, and the recurrence for $q_j$. For example, the equation for the input $a(b,\,i)$ is now $c_{P(b)}\sspeq \frac{1}{2}\,(b\sspm q_{P(b)+1}) \sspeq \frac{1}{2}\,(b\sspm q_{1})\sspeq q_0\sspeq a(b,\,i)$ from the recurrence $q_1\sspeq \abs{b\sspm 2\,q_0}\sspeq b\sspm 2\,a(b,\,i)$, because $RRS^*(b)\sspni a(b,\,i) \sspleq \frac{b-1}{2}$.  
\end{proof}
\begin{example}
$b\sspeq 63$, $i\sspeq 3$, $P \sspeq 6 \sspeq pes$, $ SBB(63,\,3)\sspeq (11,\,41,\,19,\,25,\,13,\,37)$:\pn
$c_1\sspeq \frac{1}{2}\,(63\sspm q_2)\sspeq \frac{1}{2}\,(63\sspm 19)\sspeq 22$, ...,\, $c_5\sspeq \frac{1}{2}\,(63\sspm q_6)\sspeq \frac{1}{2}\,(63\sspm q_0)\sspeq 26$, $c_6\sspeq \frac{1}{2}\,(63\sspm q_1)\sspeq 11$, hence $MDS(63,\,3)\sspeq [22,\,19,\,25,\,13,\,26,\,11]$.
\end{example}
\section{Acknowledgments}
\psn
This paper would not have been written without {\sl Gary W. Adamson}'s suggestion to study the coach system of {\sl Hilton} and {\sl Pedersen}, the book by {\sl Schick}, and the paper by {\sl Br\"andli} and {\sl Beyne}. He also introduced the author to his modified doubling sequence and his paper with {\sl Kappraff} on iterations. We had an extensive e-mail exchange. The author is very grateful to him. \pn
He also likes to thank {\sl Carl Schick}  and {\sl Gerold Br\"andli} for an inspiring e-mail exchange.

\pbn\pbn
\hrule
\pbn
AMS MSC numbers:  11B37, 11A07, 05C45. \psn
Keywords: Coach system, modified modular congruence, diagonals in regular $(2\,n)$-gons, cycles of integers, polynomial iteration, {\sl Euler} tours.
\pbn
\hrule
\pbn
Concerned with OEIS sequences:\pn
\seqnum{A000010},\, \seqnum{A000265},\, \seqnum{A001622},\, \seqnum{A003558},\, \seqnum{A023022},\, \seqnum{A038566},\, \seqnum{A053120},\, \seqnum{A055034},\, \seqnum{A065942},\, \seqnum{A082375},\, \seqnum{A082654},\, \seqnum{A127672},\, \seqnum{A135303},\, \seqnum{A187360},\, \seqnum{A216319},\, \seqnum{A216371},\, \seqnum{A232624},\, \seqnum{A268923},\,  \seqnum{A332433},\, \seqnum{A332434},\, \seqnum{A332435},\, \seqnum{A332436},\, \seqnum{A332437},\, \seqnum{A332439},\, \seqnum{A333848},\, \seqnum{A333849},\, \seqnum{A333850},\, \seqnum{A333851},\, \seqnum{A333853},\,  \seqnum{A333855},\, \seqnum{A334430}.
\pbn
\hrule
\vfill
\eject
\begin{landscape}
\begin{center}
{\large {\bf Table 1: $\bf P^*(b, x)$\ \text{for}\  $\bf b\sspeq 2\,n\sspp 1,\  n\sspeq 1,2,...,10$,\ \text {with}\  $\bf \boldsymbol{\rho}\sspeq \boldsymbol{\rho}(b)$}}
\end {center}
\begin{center}  
\begin{tabular}{|l|l|l|}\hline
\bf n & \bf b &\hskip 3cm $\bf P^*(b,x)$ \\ 
\hline\hline
$\bf 1$ & $\bf 3$  &  $ x-1 $   \\        
\hline
$\bf 2$ & $\bf 5$   &  $ x^2+(1-2\,\rho)\,x+1$      \\ 
\hline
$\bf 3$ & $\bf 7$   &  $ x^3+(3-2\,\rho^2)\,x^2+2\,\rho\,x-1 $    \\    
\hline
$\bf 4$ & $\bf 9$   &  $x^3-2\,\rho\,x^2+(-3+2\,\rho^2)\,x-1$  \\
\hline
$\bf 5$  & $\bf 11$  & $x^5+(-1+6\,\rho^2-2\,\rho^4)\,x^4+(-4-2\,\rho+2\,\rho^2+2\,\rho^3)\,x^3$ \\
&&$+(1+8\,\rho-8\,\rho^2-4\,\rho^3+2\,\rho^4)\,x^2+(3+4\,\rho-12\,\rho^2-2\,\rho^3+4\,\rho^4)\,x-1  $ \\    
\hline
$\bf 6$  & $\bf 13$  & $ x^6+(1-6 \rho+8 \rho^3-2 \rho^5) x^5+(-3-2 \rho-4 \rho^2+2 \rho^3+2 \rho^4) x^4 $      \\  
&&$+(-6+14\,\rho+10\,\rho^2-12\,\rho^3-4\,\rho^4+2\,\rho^5)\,x^3+(2+16\,\rho-24\,\rho^3+6\,\rho^5)\,x^2+(1-14\,\rho+10\,\rho^3-2\,\rho^5)\,x+1 $  \\
\hline
$\bf 7$ & $\bf 15$   & $x^4+(3-8\,\rho-2\,\rho^2+2\,\rho^3)\,x^3+(-18\,\rho+6\,\rho^3)\,x^2+(-2-10\,\rho+2\,\rho^3)\,x+1 $  \\     
\hline
$\bf 8$  & $\bf 17$  &  $x^8+(1+8\,\rho-20\,\rho^3+12\,\rho^5-2\,\rho^7)\,x^7+(-7+4\,\rho+8\,\rho^2-6\,\rho^3-8\,\rho^4+2\,\rho^5+2\,\rho^6)\,x^6$ \\
&& $ +(-26 \rho-38 \rho^2+60 \rho^3+36 \rho^4-30 \rho^5-8 \rho^6+4 \rho^7) x^5 +(9-46 \rho+124 \rho^3-6 \rho^4-80 \rho^5+2 \rho^6+14 \rho^7) x^4 $\\
&& $ +(52 \rho+32 \rho^2-142 \rho^3-24 \rho^4+84 \rho^5+4 \rho^6-14 \rho^7) x^3 +(-12+8 \rho+68 \rho^2-52 \rho^3-52 \rho^4+36 \rho^5+10 \rho^6-6 \rho^7) x^2$ \\
&&$ +(-20 \rho+6 \rho^2+20 \rho^3-2 \rho^4-4 \rho^5) x+1$\\
\hline
$\bf 9$  & $\bf 19$  & $ x^9+(-1+20 \rho^2-30 \rho^4+14 \rho^6-2 \rho^8) x^8+(-8-4 \rho+8 \rho^2+14 \rho^3-8 \rho^4-10 \rho^5+2 \rho^6+2 \rho^7) x^7$\\ 
&& $+(3+22 \rho-78 \rho^2-60 \rho^3+96 \rho^4+42 \rho^5-36 \rho^6-8 \rho^7+4 \rho^8) x^6$\\
&& $+(29-14 \rho-156 \rho^2+38 \rho^3+244 \rho^4-30 \rho^5-120 \rho^6+6 \rho^7+18 \rho^8) x^5 $\\
&& $+(-17-8 \rho+254 \rho^2-8 \rho^3-350 \rho^4+8 \rho^5+156 \rho^6-2 \rho^7-22 \rho^8) x^4$\\
&& $+(-20-30 \rho+112 \rho^2+90 \rho^3-136 \rho^4-56 \rho^5+52 \rho^6+10 \rho^7-6 \rho^8) x^3 $\\
&& $+(8+22 \rho-90 \rho^2-28 \rho^3+84 \rho^4+20 \rho^5-24 \rho^6-4 \rho^7+2 \rho^8) x^2$\\ 
&& $+(5+10 \rho-54 \rho^2-10 \rho^3+54 \rho^4+2 \rho^5-18 \rho^6+2 \rho^8) x-1$\\
\hline
$\bf 10$  & $\bf 21$ & $ x^6+(-1-12 \rho+10 \rho^3-2 \rho^5) x^5+(-10-26 \rho-2 \rho^2+16 \rho^3+2 \rho^4-2 \rho^5) x^4 $   \\
&&$+(36 \rho+16 \rho^2-26 \rho^3-6 \rho^4+4 \rho^5) x^3+(18+60 \rho-10 \rho^2-44 \rho^3+2 \rho^4+8 \rho^5) x^2+(10+10 \rho-18 \rho^2-4 \rho^3+4 \rho^4) x+1 $\\
\hline
$\hskip .2cm\bf \vdots$&\\
\hline
\end{tabular}
\end{center}
\end{landscape}
\vfill
\eject
\begin{center}
{\large {\bf Table 2: $\bf MDS(b)$\ \text{for}\  $\bf b\sspeq 2\,n\sspp 1,\  n\sspeq 1,\,2,\,...,\,35$.}}
\end {center}
\begin{center}  
\begin{tabular}{|l|l|l|}\hline
\bf n & \bf b &\hskip 3cm $\bf MDS(b,\,i),\ \ i\sspeq \range{1}{c^*(b)} $\ \  \\ 
\hline\hline
$\bf 1$ & $\bf 3$  &  $ [1] $   \\        
\hline
$\bf 2$ & $\bf 5$   &  $ [2,\, 1]$      \\ 
\hline
$\bf 3$ & $\bf 7$   &  $ [2,\, 3 ,\,1] $    \\    
\hline
$\bf 4$ & $\bf 9$   &  $[2,\,4,\, 1]$  \\
\hline
$\bf 5$  & $\bf 11$  & $[2,\,4,\, 3,\, 5,\, 1]$ \\   
\hline
$\bf 6$  & $\bf 13$  & $[2,\,4,\,5,\, 3,\, 6,\, 1]$    \\  
\hline
$\bf 7$ & $\bf 15$   & $[2,\, 4,\, 7,\, 1]$  \\     
\hline
$\bf 8$  & $\bf 17$  &  $[2,\, 4,\, 8,\, 1], \ [6 ,\, 5,\, 7,\, 3] $ \\
\hline
$\bf 9$  & $\bf 19$  &  $[2,\, 4,\, 8,\, 3,\, 6,\, 7,\, 5,\, 9,\, 1]$  \\ 
\hline
$\bf 10$  & $\bf 21$ & $ [2,\, 4,\, 8,\, 5,\,10,\, 1]$   \\
\hline
$\bf 11$  & $\bf 23$ & $[2,\,4,\,8,\,7,\,9,\,5,\,10,\,3,\,6,\, 11,\,1]$ \\
\hline
$\bf 12$  & $\bf 25$ & $[2,\,4,\,8,\,9,\,7,\, 11,\, 3,\, 6,\,12,\,1]$ \\
\hline
$\bf 13$  & $\bf 27$ & $[2,\,4,\,8,\,11,\,5,\,10,\,7,\,13,\,1]$ \\
\hline
$\bf 14$  & $\bf 29$ & $[2,\, 4,\, 8,\,13,\,3,\,6,\,12,\,5,\,10,\, 9,\,11,\,7,\,14,\,1]$ \\
\hline
$\bf 15$  & $\bf 31$ & $[2,\,4,\, 8,\,15,\,1],\ [6,\,12,\,7,\,14,\,3],\ [10,\,11,\,9,\,13,\,5]$ \\
\hline
$\bf 16$  & $\bf 33$ & $[2,\,4,\,8,\,16,\,1],\ [10,\,13,\,7,\,14,\,5]$ \\
\hline
$\bf 17$  & $\bf 35$ &$ [2,\, 4,\, 8,\, 16,\, 3,\ 6,\, 12,\ 11,\, 13,\, 9, \,17,\, 1] $ \\
\hline
$\bf 18$  & $\bf 37$ & $[2,\,4,\,8,\,16,\,5,\,10,\,17,\,3,\,6,\,12,\,13,\,11,\,15,\,7,\,14,\,9,\,18,\,1]$ \\
\hline
$\bf 19$  & $\bf 39$ & $[2,\,4,\,8,\,16,\,7,\,14,\,11,\,17,\,5,\, 10,\,19,\,1]$ \\
\hline
$\bf 20$  & $\bf 41$ & $[2,\,4,\,8,\,16,\,9,\,18,\,5,\,10,\,20,\,  1],\ [6,\,12,\,17,\,7,\,14,\,13,\,15,\,11,\,19,\,3]$ \\
\hline
$\bf 21$  & $\bf 43$ & $[2, 4, 8, 16, 11, 21, 1],\ [6, 12, 19, 5, 10, 20, 3], [14, 15, 13, 17, 9, 18, 7]$ \\
\hline
$\bf 22$  & $\bf 45$ & $[2, 4, 8, 16, 13, 19, 7, 14, 17, 11, 22, 1]$ \\ 
\hline
$\bf 23$  & $\bf 47$ & $[2, 4, 8, 16, 15, 17, 13, 21, 5, 10, 20, 7, 14, 19, 9, 18, 11, 22, 3, 6, 12, 23, 1] $ \\
\hline
$\bf 24$  & $\bf 49$ & $[2, 4, 8, 16, 17, 15, 19, 11, 22, 5, 10, 20, 9, 18, 13, 23, 3, 6, 12, 24, 1]$ \\ 
\hline
$\bf 25$  & $\bf 51$ & $ [2, 4, 8, 16, 19, 13, 25, 1],\ [10, 20, 11, 22, 7, 14, 23, 5]$ \\ 
\hline
$\bf 26$  & $\bf 53$ & $[2, 4, 8, 16, 21, 11, 22, 9, 18, 17, 19, 15, 23, 7, 14, 25, 3, 6, 12, 24, 5, 10, 20, 13, 26, 1] $ \\ 
\hline
$\bf 27$  & $\bf 55$ & $[2, 4, 8, 16, 23, 9, 18, 19, 17, 21, 13, 26, 3, 6, 12, 24, 7, 14, 27, 1] $ \\ 
\hline
$\bf 28$  & $\bf 57$ & $[2, 4, 8, 16, 25, 7, 14, 28, 1],\  [10, 20, 17, 23, 11, 22, 13, 26, 5] $ \\ 
\hline
$\bf 29$  & $\bf 59$ & $[2, 4, 8, 16, 27, 5, 10, 20, 19, 21, 17, 25, 9, 18, 23, 13, 26, 7, 14, 28, 3, 6, 12, 24, 11, 22, 15, 29, 1]$ \\
\hline
$\bf 30$  & $\bf 61$ & $[2, 4, 8, 16, 29, 3, 6, 12, 24, 13, 26, 9, 18, 25, 11, 22, 17, 27, 7, 14, 28, 5, 10, 20, 21, 19, 23, 15, 30, 1]$ \\
\hline
$\bf 31$  & $\bf 63$ & $[2, 4, 8, 16, 31, 1],\ [10, 20, 23, 17, 29, 5],\ [22, 19, 25, 13, 26, 11]$ \\
\hline
$\bf 32$  & $\bf 65$ & $[2, 4, 8, 16, 32, 1],\ [6, 12, 24, 17, 31, 3],\  [14, 28, 9, 18, 29, 7],\ [22, 21, 23, 19, 27, 11]$ \\
\hline
$\bf 33$  & $\bf 67$ & $[2, 4, 8, 16, 32, 3, 6, 12, 24, 19, 29, 9, 18, 31, 5, 10, 20, 27, 13, 26, 15, 30,$\\
& & \hskip .1cm $7, 14, 28, 11, 22, 23, 21, 25, 17, 33, 1]$ \\
\hline
$\bf 34$  & $\bf 69$ & $[2, 4, 8, 16, 32, 5, 10, 20, 29, 11, 22, 25, 19, 31, 7, 14, 28, 13, 26, 17, 34, 1] $ \\
\hline
$\bf 35$  & $\bf 71$ & $ [2, 4, 8, 16, 32, 7, 14, 28, 15, 30, 11, 22, 27, 17, 34, 3, 6, 12, 24, 23, 25, 21,$\\ 
&& \hskip .1cm $29, 13, 26, 19, 33, 5, 10, 20, 31, 9, 18, 35, 1]$ \\
\hline
\hskip .2cm$\bf \vdots$&\\
\hline
\end{tabular}
\end{center}
\vfill
\eject

\end{document}